\documentclass[11pt,reqno,oneside]{amsart}

\usepackage{graphicx}
\usepackage{amssymb}
\usepackage{epstopdf}

\usepackage[a4paper, total={6in, 9.1in}]{geometry}

\usepackage{amsmath,amsfonts,amsthm,mathrsfs,amssymb,cite}
\usepackage[usenames]{color}

\usepackage{graphics}
\usepackage{framed}

\usepackage{enumerate}
\usepackage{hyperref}
\usepackage{xcolor}
\hypersetup{
  colorlinks,
  linkcolor={blue!80!black},
  urlcolor={blue!80!black},
  citecolor={blue!80!black}
}
\usepackage[capitalize,noabbrev]{cleveref}

\numberwithin{equation}{section}

\newtheorem{Theorem}{Theorem}[section]
\newtheorem{Lemma}[Theorem]{Lemma}

\newtheorem{Definition}[Theorem]{Definition}

\newtheorem{Remark}[Theorem]{Remark}
\numberwithin{equation}{section}

\title[Inverse problem for quasilinear hyperbolic equations]{Inverse problems for a quasilinear hyperbolic equation with multiple unknowns}

\author{Yan Jiang}
\address{Department of Mathematics, Jilin University, Changchun, Jilin, China.}
\email{jiangyan20@mails.jlu.edu.cn}

\author{Hongyu Liu}
\address{Department of Mathematics, City University of Hong Kong, Hong Kong SAR, China.}
\email{hongyu.liuip@gmail.com, hongyliu@cityu.edu.hk}

\author{Tianhao Ni}
\address{School of Mathematical Sciences, Zhejiang University, Hangzhou, Zhejiang, China.}
\email{12035009@zju.edu.cn}

\author{Kai Zhang}
\address{Department of Mathematics, Jilin University, Changchun, Jilin, China.}
\email{zhangkaimath@jlu.edu.cn}

\begin{document}
\maketitle

\begin{abstract}

We propose and study several inverse boundary problems associated with a quasilinear hyperbolic equation of the form ${c(x)^{-2}}\partial_t^2u=\Delta_g(u+F(x, u))+G(x, u)$ on a compact Riemannian manifold $(M, g)$ with boundary. We show that if $F(x, u)$ is monomial and $G(x, u)$ is analytic in $u$, then $F, G$ and $c$ as well as the associated initial data can be uniquely determined and reconstructed by the corresponding hyperbolic DtN (Dirichlet-to-Neumann) map.  Our work leverages the construction of proper Gaussian beam solutions for quasilinear hyperbolic PDEs as well as their intriguing applications in conjunction with light-ray transforms and stationary phase techniques for related inverse problems. The results obtained are also of practical importance in assorted of applications with nonlinear waves.

\end{abstract}

\medskip
	


\section{Introduction}\label{sect:1}

\subsection{Mathematical setup and major findings}

Let $(\mathcal{M},\overline{g})$ denote a Lorentzian manifold of dimension $1+3$ with boundary. We adopt the premise that $(\mathcal{M},\overline{g})$ possesses a  global product structure, $\mathcal{M} =[0,T]\times M$, where the metric is represented as $\bar{g}=-dt^2+g$, with $T>0$. Here, $(M,g)$ signifies a Riemannian manifold with a smooth boundary $\partial M$, and $g\in C^{10}(M)$ is the metric tensor on $M$. 

Denote $\Delta_{\overline{g}}$ as the Laplace-Beltrami operator on $(\mathcal{M},\overline{g})$. In the local coordinates $(t,x)=(t,x_1,x_2,x_3)$, we have for any $u\in C^2(\mathcal{M})$ that
\begin{align*}
\Delta_{\overline{g}}u = \left(-\partial_t^2+\Delta_g\right)u= -\partial_t^2 u+ \sum\limits_{i,j=1}^3|g|^{-1/2}\partial_{i}(|g|^{1/2}g^{ij}\partial_{j} u),
\end{align*}
where $g^{-1} := (g^{ij})_{1\leq i,j\leq 3}$, $|g| := \det g$ and $\partial_i u:=\partial u/\partial x_i$. 

In the Lorentzian manifold $(\mathcal{M},\overline{g})$, we consider the following Cauchy problem for the hyperbolic PDE:
\begin{equation}
\label{equ:Quasilinear_equation}
    \left\{
    \begin{aligned}
        &\partial^2_t u = a(x)\Delta_g (u+F(x,u))+G(x,u)\ &&\text{in}\ \mathcal{M},\\
        &u(t, x)=h(t, x)\ &&\text{on}\ \Sigma:=(0,T)\times \partial M,\\
        &u(0, x)=\varphi(x),\ \partial_t u(0,x) = \psi(x) &&\text{on}\ M,
    \end{aligned}
\right.
\end{equation}
where $a(x)$ is a strictly positive function, $F$ and $G$ are nonlinear terms to be specified later, and $h$ and $(\varphi, \psi)$ signify the boundary and initial conditions, respectively.

Next, we introduce some function spaces and notations for our subsequent use. For any integer $m\ge 0$, we set
\begin{align*}
    \mathcal{H}^{m} = \left\{h\in H^m(\Sigma): h\in H^{m-k}([0,T];H^k(\partial M)),\ \text{for }k=0,1,\cdots,m \right\},
\end{align*}
and define the energy space as
\begin{align*}
    E^m = \bigcap\limits_{k=0}^m C^k([0,T];H^{m-k}(M)),
\end{align*}
where the associated norm of the energy space $E^m$ is given by
\begin{align*}
    \|u\|_{E^m} = \sup\limits_{t\in [0,T]} \sum\limits_{k=0}^m \left\|\partial_t^k u(t,\cdot)\right\|_{H^{m-k}(M)},\,\forall\, u\in E^m.
\end{align*}
Furthermore, $E^m$ is a Banach algebra (cf.\cite{B08}), that is
\begin{align*}
    \|uv\|_{E^m}\leq C_m\|u\|_{E^m}\|v\|_{E^m},\, \forall\, u,v \in E^m.
\end{align*}

We now introduce the hyperbolic Dirichlet-to-Neumann (DtN) map as follows:
\begin{equation}
    \label{equ:DtN_map}
    \begin{aligned}
\Lambda_{a, F, G, \varphi,\psi}:\mathcal{H}^{m+1}& \to H^{m}(\Sigma),
\\h & \to  \partial_{\nu_g}(u+F(x,u))\big|_{\Sigma},
\end{aligned}
\end{equation}
where $u$ solves equation \eqref{equ:Quasilinear_equation} subject to the Dirichlet boundary input $h$. Here, $\nu$ represents the outward unit normal vector to $\partial M$, and $\partial_{\nu_g}$ is the conormal derivative.
In Section \ref{sec:proof_of_direct_problem}, we will establish the well-posedness of the DtN map (\ref{equ:DtN_map}) in a proper setup. The DtN map encodes the measurement data for our subsequent inverse problem study. It can be seen that knowing the DtN map is equivalent to knowing all the possible boundary pairs $(u, \partial_{\nu_g}(u+F(x,u)))$ where $u\in E^m$
is a solution to \eqref{equ:Quasilinear_equation}.

We proceed to introduce the inverse problem of our study which is concerned with recovering $a, F, G$ and $\varphi, \psi$ by knowledge of $\Lambda_{a, F, G, \varphi, \psi}$. That is, 
\begin{equation}\label{eq:ip1}
\Lambda_{a, F, G, \varphi, \psi}\longrightarrow a, F, G, \varphi, \psi. 
\end{equation}
In this article, we mainly study the unique identifiability and reconstruction issues for the inverse problem \eqref{eq:ip1}. For the unique identifiability issue, we derive sufficient conditions in a general setup to ensure that one can uniquely determine $(a, F, G, \varphi, \psi)$ by knowledge of $\Lambda_{a, F, G, \varphi, \psi}$; that is, the correspondence between $(a, F, G, \varphi, \psi)$ and $\Lambda_{a, F, G, \varphi, \psi}$ is one-to-one. The main results are presented in Theorems~\ref{Theo:IP11} and \ref{Theo:IP21}, which leverages the construction of proper Gaussian beam solutions for quasilinear hyperbolic PDEs and the delicate use of light-ray transforms. For the unique reconstruction issue, we establish explicit formulas in recovering $F$ and $G$ which leverages the utilization of proper Gaussian beams in combination with the stationary phase techniques. The main results can be briefly summarised as follows:

\begin{Theorem}\label{thm:sum1}
Suppose that $a, F, G$ and $\varphi, \psi$ belong to the admissible class in Definition~\ref{def:adm con}. If Assumption~I in Section~\ref{sect:3} is fulfilled, then $\Lambda_{a, F, G, \varphi, \psi}$ uniquely determines $a, F, G$ and $\varphi, \psi$. That is, the correspondence between $\Lambda_{a, F, G, \varphi, \psi}$ and $(a, F, G, \varphi, \psi)$ is one-to-one. On the other hand, if Assumption~II in Section~\ref{sect:3} is fulfilled and $a$ is a-priori given, then $\Lambda_{F, G, \varphi, \psi}$ uniquely determines $a, F, G$ and $\varphi, \psi$, and moreover there are explicit reconstruction formulas for $F$ and $G$. 
\end{Theorem}

Finally, we would like to point out that it is straightforward to see that if one considers the hyperbolic equation of the form ${c(x)^{-2}}\partial_t^2u=\Delta_g(u+F(x, u))+G(x, u)$ with a strictly positive $c(x)$, then the hyperbolic DtN map can uniquely determine $c, F, G$ as well as the associated initial data. 

%

\subsection{Background discussion and technical novelties}

Quasilinear hyperbolic equations find wide-ranging applications in various fields, including fluid dynamics, gas dynamics, elasticity, electromagnetism, and general relativity. They provided mathematical frameworks for modelling wave propagation in different media and understanding phenomena such as shock waves, sound waves, and electromagnetic waves. The term ``quasilinear'' was coined to describe equations where the coefficients depend nonlinearly on the unknown function and its derivatives. The systematic study of quasilinear hyperbolic equations and their properties, including quasilinear variants, took shape in the early to mid-20th century. We refer to \cite{L51,MN98,PS20,HY23,MR23, JLNZ23, CW95,MCW94} for several quasilinear hyperbolic equations that are related to the current study. In particular, we note the nonlinear progressive equation from \cite{JLNZ23, CW95,MCW94} to describe the infrasonic wave propagation, which forms one of the major motivating problems for our study and will be discussed in more details in Section~\ref{sect:2}.

The inverse problem for quasilinear hyperbolic equations deals with determining unknown parameters or functions in the equation based on observed data. This problem has applications in various fields including medical imaging, geophysics, nondestructive testing, and material science. In particular, we refer to \cite{JLNZ23} on applications related to infrasonic inversions. In inverse problems involving nonlinear equations, it is a common practice to linearize the equation and utilize solutions for the linearized ones. This approach is evident in various scenarios such as the treatment of semilinear parabolic equations~\cite{I93}, the nonlinear Calder\'on problem~\cite{SunU04}, and the nonlinear Navier-Stokes equation~\cite{LW07}. However, since the publication of \cite{KLU18}, there has been a shift towards leveraging the nonlinearity itself to tackle certain inverse problems. The key is the interaction of multiple waves, generating new waves that yield information beyond what is attainable through the linearized equation. Meanwhile, significant tools have been developed to address inverse problems for nonlinear equations, including the light ray transforms and the stationary phase techniques along with the Gaussian beam solutions; see  \cite{SU04, FO20, FIKO21, UZ21b,L1,L2,L3} and the references cited therein. 

In this paper, we delve into inverse problems for quasilinear hyperbolic equations with multiple unknowns, thus significantly extending our findings from \cite{JLNZ23} to a much broad scope. We derive both unique identifiability results and reconstruction formulas for the proposed inverse problems.  Our work leverages the construction of proper Gaussian beam solutions for quasilinear hyperbolic PDEs as well as their intriguing applications in conjunction with light-ray transforms and stationary phase techniques for related inverse problems.


The rest of the paper is oganized as follows. In Section~\ref{sect:2}, we establish the local well-posedness of the forward problem. Section~\ref{sect:3} presents the geometric conditions for the inverse problems and summarizes the main theorems. In Section~\ref{sect:4}, we provide the proofs of the main theorems.

\section{Local well-posedness of the forward problem}\label{sect:2}

\subsection{Technical preparations }

Consider the hyperbolic system \eqref{equ:Quasilinear_equation}. We introduce the following technical conditions on $F, G$ and $\varphi, \psi$ for our subsequent study.
 
\begin{Definition}\label{def:adm con}
The functions $a, F$, $G$, $h$, $\varphi$, and $\psi$ are said to satisfy the admissible conditions if the following criteria are fulfilled:
\begin{enumerate}
\item[(1)] $a(x)$ is a strictly positive smooth function. $F$ has the following monomial form:
\begin{equation}\label{eq:mam1}
F(x,y) =  b_n(x)y^n,
\end{equation}
where $b_n \in C^{\infty}({M})$ with $n\in\mathbb{N}$ and $n\ge 2$. $G$ possesses the following Taylor's series representation:
\begin{equation}\label{eq:g ck}
G(x,y) = \sum\limits_{k=1}^{\infty} c_k(x){y^k}/{k!},
\end{equation}
where $c_k \in C^{\infty}({M})$ for $k=1,2,\cdots$. Furthermore, we assume that there exists a positive constant $C$ such that
\begin{equation}\label{eq:mam2}
    \| b_n\|_{L^{\infty}({M})}\leq C,\,\,\,\|c_k\|_{L^{\infty}({M})}\leq C
\end{equation}
for all $k\in\mathbb{N}$. 

\item[(2)] On the boundary $\partial M$, the triple $(h,\varphi,\psi)$ satisfy the following compatibility conditions:
\begin{equation}\label{eq:adm3}
\left\{
\begin{aligned}
&h(0,x)=\varphi(x),\\
&\partial_t h(0,x) = \psi(x),\\
&\partial_{t}^2h(0,x) = a(x)\Delta_{g}(h(0,x)+F(x,\varphi(x))+G(x,\varphi(x)),\\
&\cdots\\
&\mbox{up to the m-th order derivative of $h$},
\end{aligned}
\right.
\end{equation}
where $h\in H^{m}(\Sigma)$ for a positive integer $m$, and $H^{m}(\Sigma)$ denotes the Sobolev space. 

\item[(3)] For homogeneous boundary and initial conditions, there exists a unique solution $u\equiv0$ that satisfies the hyperbolic system \eqref{equ:Quasilinear_equation}.
\end{enumerate}
\end{Definition}

Two remarks are in order regarding the assumptions on the function parameters in our study. First, we note that the nonlinear progressive equation (NPE) belongs to equation \eqref{equ:Quasilinear_equation} that fulfils \eqref{eq:mam1}, which describes the propagation of infrasound \cite{CW95,MCW94}.
In fact, the NPE takes the following form:
\begin{equation}
\label{equ:infrasound_equation}
\frac{1}{c_0(x)^{2}}\partial_t^2u = \Delta \left(u+\frac{c_2(x)}{c_0(x)^2}u^2\right),
\end{equation}
where $c_0(x)$ signifies the speed of the wave propagation and $c_2(x)$ is a nonlinear coefficient determined by the sound pressure and the atmospheric density. Hence, it is of the form \eqref{eq:mam1}--\eqref{eq:g ck} with $a(x)=c_0(x)^2$, $F=\frac{c_2(x)}{c_0(x)^2}u^2$ and $G=0$ under the Euclidean metric. Second, in \eqref{eq:mam1} and \eqref{eq:g ck}, we assume that $F$ and $G$ are independent of $t$, namely the time variable. This technical condition will be required only for the inverse problem study. For the forward problem, we can establish a more general result with time-dependent $F$ and $G$, which is of independent interest for its own sake. In such a case, $b_n(x)$ and $c_k(x)$ in \eqref{eq:mam1}--\eqref{eq:mam2} are replaced by $b_n(x, t)$ and $c_k(x, t)$ and $M$ is replaced by $\mathcal{M}$, and $F(x, \varphi)$ and $G(x,\varphi)$ in \eqref{eq:adm3} are replaced by $F(t=0, x, \varphi)$ and $G(t=0, x,\varphi)$ accordingly. 

We proceed to establish the local well-posedness of the forward problem \eqref{equ:Quasilinear_equation} with time-dependent $F$ and $G$ as described above. Our approach leverages the implicit function theorem in Banach spaces in conjunction with the following lemma (\cite{P87}):

\begin{Lemma}\label{lemma:analytic}
Consider a mapping $f: U \to F$, where $U\subset E$ and both $E$ and $F$ are complex Banach spaces. The following two conditions are equivalent:
\begin{enumerate}
\item $f$ is analytic on $U$;

\item $f$ is locally bounded and weakly analytic on $U$, where weakly analyticity means that for any $x_0, x \in U$, the mapping $\lambda \mapsto f(x_0 + \lambda x)$ is analytic in a neighborhood of the origin in the complex plane.
\end{enumerate}
\end{Lemma}

We proceed to present an auxiliary result by considering the following linear hyperbolic equation:
\begin{equation}
\label{equ:linear_hyperbolic_equation}
\left\{
    \begin{aligned}
    &\partial_t^2 v-a(x)\Delta_{g}v+q(t,x)v=p(t,x),\ &&\text{on}\ \mathcal{M},\\
        &v(t,x)=h,\ &&\text{on}\ \Sigma,\\
        &v(0,x)=\varphi(x),\ \partial_t v(0,x) = \psi(x),&&\text{on}\ M,
\end{aligned}\right.
\end{equation}
where $p$ and $q$ are given functions. The following lemma in \cite{KKL01} provides the global well-posedness of the hyperbolic equation \eqref{equ:linear_hyperbolic_equation}.
\begin{Lemma}
\label{lemma: wellposed linear hyperbolic equation}
Let $m$ be a positive integer and $T>0$. Suppose $h\in \mathcal{H}^{m+1}(\Sigma)$, $\varphi\in H^{m+1}(M)$, $\psi\in H^{m}(M)$, $\partial_t^j p(\cdot,x)\in L^1([0,T];H^{m-k}(M))$ ($j=0,1,\cdots,m$), and the ensuing compatibility conditions are satisfied on $\partial M$:
\begin{equation*}
\left\{
\begin{aligned}
&h(0,x) = \varphi(x),\\
&\partial_t h(0,x) = \psi(x),\\
&\partial_t^2h(0,x)-a(x)\Delta_g \varphi(x)+q(0,x)\varphi(x)=p(0,x),\\
&\partial_t^3h(0,x)-a(x)\Delta_g \psi(x)+q(0,x)\psi(x)+\partial_t q(0,x)\varphi(x)=\partial_t p(0,x),\\
&\cdots\\
&\text{up to the m-$th$ order temporal derivative of $h$}.
\end{aligned}\right.
\end{equation*}
Then the hyperbolic equation \eqref{equ:linear_hyperbolic_equation} has a unique solution $v\in E^{m+1}$, and $\partial_{\nu} v \in H^m(\Sigma)$.
\end{Lemma}

We are in a position to present the local well-posedness of the hyperbolic system \eqref{equ:Quasilinear_equation}.
\begin{Theorem}[Local well-posedness]
\label{theorem:local well-posedness}
Under the admissibility conditions in Definition~\ref{def:adm con}, for any sufficiently small $\epsilon_0$ and positive integer $m$, and assuming the boundary and initial values belong to the set
\begin{equation}\label{eq:setn1}
\mathcal{N} = \left\{h\in \mathcal{H}^{m+1}(\Sigma),\varphi\in H^{m+1}(M),\psi\in H^m(M) : \|h\|_{H^{m+1}(\Sigma)}+\|\varphi\|_{H^{m+1}(M)}+\|\psi\|_{H^m(M)}<\epsilon_0\right\},
\end{equation}
there exists a unique solution $u\in E^{m+1}$ such that $\partial_{\nu} u\in H^m(\Sigma)$ and
\begin{align*}
\|u\|_{E^{m+1}}+\|\partial_{\nu} u\|_{H^m(\Sigma)}\leq C\epsilon_0,
\end{align*}
where $C$ is a positive constant independent of $u$, $h$, $\varphi$, and $\psi$.
\end{Theorem}
\begin{Remark}
\label{rmk:sobolev_embedding}
According to the Sobolev embedding theorem, we observe that $u\in C^{\ell}(\mathcal{M})$ holds true when $m>\frac{3}{2}+\ell$ for any non-negative integer $\ell$ (\cite{A03}).
\end{Remark}

\begin{proof}[Proof of Theorem~\ref{theorem:local well-posedness}]

The proof proceeds in three steps as follows. 

\textbf{Step 1.} Consider the following spaces
\begin{align*}
    U_0 = &\left\{p\in L^1([0,T];H^m(M)):\partial_t^j p(\cdot,x)\in L^1([0,T];H^{m-j}(M)),\forall j=1,\cdots,m \right\},\\
    U_1 = &\mathcal{H}^{m+1}\times H^{m+1}(M)\times H^{m}(M),\\
    U_2 = &\left\{u\in E^{m+1}:u\big|_{\Sigma}\in \mathcal{H}^{m+1},\partial_{\nu}u\in H^{m}(\Sigma),u(0,\cdot)\in H^{m+1}(M),\partial_t u(0,\cdot)\in H^{m}(M),\right.\\
    &\left. \partial_t^2 u-\Delta_g (u+F(t,x,u))\in U_0\right\},
\end{align*}
and set the error operator as
\[
  E:U_1\times U_2\to U_0\times U_1,
\]
with
\[
    E(u;h,\varphi,\psi) = [E_e; (u(t,x)-h)|_{\Sigma},u(0,x)-\varphi(x),\partial_t u(0,x)-\psi(x)].
\]
Here, $u\in U_1,(h,\varphi,\psi)\in U_2$, and $E_e: U_1\times U_2\to U_0$ is given by
\[
E_e(u) = \partial_t^2 u-a(x)\Delta_{g}\left(u+F(t,x,u)\right)-G(t,x,u).
\]

We first establish the well-definedness of $E$, implying that we only need to demonstrate that for any $u\in U_2$, it yields that
\[
E_e(u) \in U_0 \quad \mbox{or} \quad G(t,x,u)= \sum\limits_{k=1}^{\infty} c_k(t,x)\frac{u^k}{k!} \in U_0.
\]

On one hand, for $k\geq 1$, it follows from the definition of $G$ and the CBS inequality that
\begin{align}\label{equ:G est}
    \|G(\cdot,\cdot,u)\|_{L^1([0,T];H^m(M))} \leq &\sum\limits_{k=1}^{\infty}\frac{1}{k!} \|c_k\|_{L^1([0,T];H^m(M))}\|u\|_{L^1([0,T];H^m(M))}^k \nonumber\\
    \leq &\sum\limits_{k=1}^{\infty} \frac{C_1^k}{k!} \|c_k\|_{L^1([0,T];H^m(M))}\|u\|_{E^m}^k \leq Ce^{C_1\|u\|_{E^m}},
\end{align}
where $C$ denotes the upper bound under the admissible conditions and $C_1$ is a positive number depending on $T$.

On the other hand, for $j=1$ and $k\geq 1$, we have
\begin{align*}
   \|\partial_t\left(c_ku^k\right)\|_{L^1([0,T];H^{m-1}(M))} \leq &\|u^k\partial_t c_k\|_{L^1([0,T];H^{m-1}(M))}+k\|u^{k-1}c_k\partial_tu\|_{L^1([0,T];H^{m-1}(M))}\\
   \leq & C (1+k)C_1^k\|u\|^k_{E^{m-1}},
\end{align*}
which readily gives that
\[
 \|\partial_t G(\cdot,\cdot,u)\|_{L^1([0,T];H^{m-1}(M))}  \leq  C\sum\limits_{k=2}^{\infty}\frac{1+k}{k!}C_1^k \|u\|^k_{E^{m-1}}
    \leq 2Ce^{C_1\|u\|_{E^{m-1}}}.
  \]
This implies that $\partial_t G(\cdot,\cdot,u)\in L^1([0,T];H^{m-1}(M))$, and likewise, $\partial^{j}_t G(\cdot,\cdot,u)\in L^1([0,T];H^{m-1}(M))$ for $j=2,\cdots,m$. Coupled with (\ref{equ:G est}), this suggests that $G(\cdot,\cdot,u) \in U_0$.

\textbf{Step 2.} The admissibility condition on $G$ suggests that for any $u_1,u_2\in U_1$, there exists
\[
   G(t,x,u_1+\lambda u_2)=\sum\limits_{k=1}^{\infty} c_k(t,x)\frac{(u_1+\lambda u_2)^k}{k!},
\]
which indicates that the nonlinear operator $G(\cdot,\cdot,u)$ possesses weak analyticity. In addition, the boundedness of $c_k$, $u_1$, and $u_2$ ensures that $G(\cdot,\cdot,u)$ is locally bounded, as well as yields the same result for $F(\cdot,\cdot,u)$. Consequently, $E$ emerges as locally bounded and weakly analytic. Coupled with Lemma \ref{lemma:analytic}, this implies the analyticity of $E$ within $U_1\times U_2$. Utilizing the implicit function theorem, we establish the well-posedness of \eqref{equ:Quasilinear_equation}.

\textbf{Step 3.} We next deduce that $E(0;0,0,0)=0$. We compute the Fr\'echet derivative of $E$ with respect to $u$ evaluated at 0:
\begin{align*}
    E_u(0;0,0,0)v = (\partial_t^2 v-a(x)\Delta_{g}v-c_1(t,x)v,v|_{\Sigma},v(0,x),\partial_t v(0,x)),\,\,\, \forall v\in U_1.
\end{align*}
Hence, the operator $E_u(0;0,0,0)$ is an isomorphism if and only if the subsequent system possesses a unique solution $v\in E^{m+1}$:
\begin{equation*}
\left\{
    \begin{aligned}
    &\partial_t^2 v-a(x)\Delta_{g}v+c_1(t,x)v=p(t,x),\ &&\text{on}\ \mathcal{M},\\
        &v(t,x)=h,\ &&\text{on}\ \Sigma,\\
        &v(0,x)=\varphi(x),\ \partial_t v(0,x) = \psi(x),&&\text{on}\ M,
\end{aligned}\right.
\end{equation*}
for any $p\in U_0$ and $(h,\varphi,\psi)\in U_1$. The well-posedness of the above system can be inferred from Lemma~\ref{lemma: wellposed linear hyperbolic equation}. Hence, we know that $E_u(0;0,0,0)$ forms an isomorphism.

Given the analyticity of $E$ (as verified in Step 2) and the fact that $E_u(0;0,0,0)$ is a linear isomorphism, we can deduce from Lemma \ref{lemma:analytic} the existence of a positive constant $\epsilon_0$ such that
\[
    \mathcal{N} = \{(h,\varphi,\psi)\in U_1 : \|h\|_{H^{m+1}(\Sigma)}+\|\varphi\|_{H^m(M)}+\|\psi\|_{H^m(M)}<\epsilon_0\},
\]
and a $C^{\infty}$  Fr\'echet differentiable mapping $\mathcal{P}:U_1\to U_2$ such that
\[
E(\mathcal{P}(h,\varphi,\psi);h,\varphi,\psi)=(0,0,0,0)
\]
with the energy estimate of $u=\mathcal{P}(h,\varphi,\psi)$ as
\[
    \|u\|_{E^{m+1}}+\|\partial_{\nu} u\|_{H^m(\Sigma)}\leq C\epsilon_0,
\]
where $C$ is a positive constant independent of $u$, $h$, $\varphi$ and $\psi$.

The proof is complete.

\end{proof}

\section{Statement of main unique identifiability results for the inverse problem}\label{sect:3}


We introduce some preliminary knowledge on geodesics for the subsequent use.
Let $SM=\{(y,v)\in TM:y\in M, v\in T_{y}M, |v|=1\}$, where $TM$ denotes the tangent bundle and $T_{y}M$ represents the tangent space at $y$. For any $(y,v)\in SM$, we define $\gamma(\cdot,y,v)=\exp_{y}(v)$ as the unit speed geodesic starting at the point $y$ in the direction $v$.
Following \cite{FIKO21}, we define
\begin{align*}
    & t_{exit}(y,v):=\inf\{t>0:\gamma(t,y,v)\in \partial M,\dot{\gamma}(t,y,v)\notin T_{\gamma({t,y,v})}\partial M\},\\
    & \partial_{-}{SM}:=\{(y,v)\in SM:y\in\partial M,\langle v,\nu(y)\rangle_g<0,t_{exit}(y,v)<\infty \},
\end{align*}
where $\dot{\gamma}$ represents the derivative of $\gamma$ with respect to $t$ and $\nu$ signifies the outward unit normal vector at $y\in\partial M$. For simplicity, we denote $t_{exit}(y,v)$ as $t_{exit}$ in short. Then the maximal interval is given by $I:=(0,t_{exit})$.

For any $f\in C^{\infty}(\mathcal{M})$, we define $\nabla^{\overline{g}}{f}=\overline{g}^{ij}\partial_{x_j}f$. A curve $\alpha \in \mathcal{M}$ is called a null geodesic if $\langle\nabla^{\overline{g}}f,\dot{\alpha}\rangle_{\overline{g}}=0$ and $\langle \dot{\alpha},\dot{\alpha} \rangle_{\overline{g}}=0$, where $\langle\cdot,\cdot\rangle_{\overline{g}}$ denotes the inner product on $\mathcal{M}$ induced by the metric $\overline{g}$. Utilizing the product structure of $\mathcal{M}$, we can parameterize maximal null geodesics as $\alpha(t)=(t_0+t,\gamma(t))$ for $t\in I$, where $t_0\in \mathbb{R}$ and $\gamma$ represents a unit-speed maximal geodesic in $M$ (cf.\cite{Car17}). For notational convenience, we assume $t_0=0$.

For any $x\in M$, we define $D_g(x):=\sup\{t_{exit}(y,v):(y,v)\in \partial_{-}SM, x\in \gamma(\cdot;y,v)\}$, and let $D_g(M):=\sup\{D_g(x)|x\in M\}$. For $T>2D_g(M)$, we define
\begin{equation}
\label{equ:def_E}
    \mathcal{E}:=\{(t,x)\in \mathcal{M}:D_g(x)<t<T-D_g(x)\}.
\end{equation}

Next, we introduce two geometric assumptions which are necessary for our inverse problem study, and we also refer to \cite{DE17,WZ19} for related background discussion on these two technical conditions. Define the diameter of $M$ as diam$(M)$:
\begin{align*}
    \mbox{diam}(M) = \sup\limits_{p,q\in M}\left\{d_{\gamma}:\gamma\text{ is a geodesic connecting } p\text{ and }q\right\},
\end{align*}
where $d_{\gamma}$ denotes the length of the geodesic $\gamma$.

\textbf{Assumption I.}~Assume that $T>2D_g(M)$ and $M$ is a compact, smooth Riemannian manifold without conjugate points. The coefficients $b_n$ in \eqref{eq:mam1} and $c_k$ in \eqref{eq:g ck} for $k\geq 2$ in Definition~\ref{def:adm con} depend solely on the spatial variable $x$, and the support of the coefficient $c_1$ is contained in a subset $\mathcal{E}\Subset M$.

\textbf{Assumption II.} Assume $T>2\mbox{diam}(M)$, $\partial M$ is strictly convex with respect to $g$, and at least one of the following conditions holds:
\begin{itemize}
    \item [1.] $(M,g)$ has no conjugate points;
    \item [2.] $(M,g)$ satisfies the foliation condition.
\end{itemize}

The Assumption I will be necessary to establish the unique identifiability result via the ray transform, whereas Assumption II will be required for inversion using the stationary phase techniques. 

We are in a position present the main unique identifiability results for the inverse problem \eqref{eq:ip1} with admissible $a, F, G$ and $\varphi, \psi$. We will distinguish between two cases whether $(\varphi, \psi)\equiv (0, 0)$ or not. If $(\varphi, \psi)\equiv (0, 0)$, we simply write $\Lambda_{a, F, G, \varphi, \psi}$ as $\Lambda_{a, F, G}$ and in such a case, $\Lambda_{a, F, G}(h)$ encodes the measurement data generated by actively injecting the boundary input $h$. Hence, it is referred to as the active measurement in the literature. If $(\varphi, \psi)\equiv\hspace*{-3.5mm}\backslash\ (0, 0)$, as the unknown sources, they will also generate wave signals for the boundary observer to passively receive. Hence in such a case, $\Lambda_{a, F, G, \varphi, \psi}(h)$ encodes both the passive measurement data generated by the unknown sources $\varphi$ and $\psi$, and the active measurement generated by the boundary input $h$. It is noted that inverse problems with passive measurements constitute a challenging and longstanding topic in the field of inverse problems (cf. \cite{L4}).

\begin{Theorem}
\label{Theo:IP11}
Let $(a_j, F_j, G_j)$, $j=1,2$, be two admissible configurations. Suppose that Assumption I holds. If
\begin{equation}\label{eq:uu1}
    \Lambda_{a_1,F_1,G_1}(h) = \Lambda_{a_2,F_2,G_2}(h),\, \forall\, h\in \mathcal{N},
\end{equation}
where $\mathcal{N}$ is defined in \eqref{eq:setn1}, then one has
\begin{equation}\label{equ:FG}
    (a_1,F_1, G_1)=(a_2,F_2, G_2).
\end{equation}
That is, the hyperbolic DtN map $\Lambda_{a,F,G}$ uniquely determines $(a,F,G)$.
\end{Theorem}

\begin{Theorem}
\label{Theo:IP12}
Let $(a, F_j, G_j)$, $j=1,2$, be two admissible configurations. Suppose that Assumption II holds and the coefficient $c_1$ in \eqref{eq:g ck} as well as $a$ are a-priori given. If
\begin{equation}\label{eq:uu2}
     \Lambda_{a,F_1,G_1}(h) = \Lambda_{a,F_2,G_2}(h),\, \forall\, h\in \mathcal{N},
\end{equation}
then one has
\begin{equation}\label{eq:FG2}
(F_1, G_1)=(F_2, G_2). 
\end{equation}
That is, the hyperbolic DtN map $\Lambda_{a,F,G}$ uniquely determines $(F, G)$ if $a$ is a priori given. Moreover, $F$ can be uniquely reconstructed according to the formulas \eqref{equ:second order phase stationary lemma}--\eqref{eq:rb1}, and $G$ can be uniquely reconstructed according to the formulas \eqref{eq:rc1} --\eqref{equ:second order integral phase}. 
\end{Theorem}


\begin{Theorem}
\label{Theo:IP21}
Let $(a_j, F_j, G_j, \varphi_j, \psi_j)$, $j=1,2$, be two admissible configurations. Suppose that Assumption I holds. If
\begin{equation}\label{eq:uu3}
    \Lambda_{a_1, F_1, G_1, \varphi_1,\psi_1}(h) = \Lambda_{a_2, F_2, G_2, \varphi_2,\psi_2}(h),\, \forall\, h\in \mathcal{H}^{m+1}(\Sigma),
\end{equation}
then one has
\begin{equation}\label{equ:FGphi}
    (a_1, F_1, G_1, \varphi_1,\psi_1)=(a_2, F_2, G_2, \varphi_2,\psi_2).
\end{equation}
That is, the hyperbolic DtN map $\Lambda_{\varphi,\psi,a,F,G}$ uniquely determines $(a, F, G, \varphi,\psi)$.
\end{Theorem}

\begin{Theorem}
\label{Theo:IP22}
Let $(a, F_j, G_j, \varphi_j, \psi_j)$, $j=1,2$, be two admissible configurations. Suppose that Assumption II holds and the coefficient $c_1$ in \eqref{eq:g ck} as well as $a$ are a-priori given. If
\begin{equation}\label{eq:uu4}
    \Lambda_{a, F_1, G_1, \varphi_1,\psi_1}(h) = \Lambda_{a, F_2, G_2, \varphi_2,\psi_2}(h),\, \forall\, h\in \mathcal{H}^{m+1}(\Sigma),
\end{equation}
then one has
\begin{equation}\label{eq:FGphi2}
(F_1, G_1, \varphi_1, \psi_1)=(F_2, G_2, \varphi_2, \psi_2). 
\end{equation}
Moreover, $F$ and $G$ can be uniquely reconstructed by using the same reconstruction formulas in Theorem~\ref{Theo:IP21}. 
\end{Theorem}

\section{Proofs of Theorems~\ref{Theo:IP11}--\ref{Theo:IP22}}\label{sect:4}

In this section, we provide proofs for Theorems~\ref{Theo:IP11}$-$\ref{Theo:IP22}. To ease the exposition, we shall only deal with the case $n=2$ in \eqref{eq:mam1} since the other cases with $n>2$ can be proved by straightforwardly modifying our arguments. Throughout this section, we will abbreviate $\Lambda_{\varphi,\psi,a,F,G}$ as $\Lambda$. Before delving into the proofs of the corresponding theorems, we present two pivotal tools in our analysis including the successive linearization and the Gauss beams. 

\subsection{Successive linearization}

In the proof of Theorem\,\ref{theorem:local well-posedness}, we have already established the continuous dependence of the solution $u$ on the boundary values $h$, as well as on $\varphi$ and $\psi$. Therefore, if there are certain asymptotic parameters $\epsilon=(\epsilon_1,\cdots,\epsilon_N)$ involved in the boundary input $h$, it is clear that the corresponding solution $u$ depends on these parameters in a continuous manner. From this perspective, we can introduce the successive linearization method, a.k.a. higher-order linearization method (cf.\cite{KLU18}) to reduce the nonlinear hyperbolic equation into a series of linear ones.

Let us consider the hyperbolic system \eqref{equ:Quasilinear_equation} of the following form:
\begin{equation}
\label{equ:h_expansion}
    \left\{
    \begin{aligned}
        &\partial^2_t u = a(x)\Delta_g (u+F(t,x,u))+G(t,x,u),\ &&\text{in}\ \mathcal{M},\\
        &u(t, x)=\sum\limits_{i=1}^N\epsilon_{i}h^{(i)}(t,x),\,,\ &&\text{on}\ \Sigma=(0,T)\times \partial M,\\
        &u(0, x)=0,\ \partial_t u(0,x) = 0,&&\text{on}\ M,
    \end{aligned}
\right.
\end{equation}
where $\epsilon = (\epsilon_1,\cdots,\epsilon_N)\in\mathbb{R}^N$ consisting of the asymptotic parameters $\epsilon_j$, $j=1,2,\ldots, N$ and $h^{(j)}\in \mathcal{H}^{m+1}(\Omega)$ for $j=1,\cdots,N$.
Let $u(t,x,\epsilon)$ be the corresponding solution to \eqref{equ:h_expansion}, and let $\Lambda(h,\epsilon)$ denote the hyperbolic DtN map associated with $u(t,x,\epsilon)$; see \eqref{equ:DtN_map} for the case with general boundary inputs. 

To ensure the well-posedness of the direct problem under consideration, according to Theorem \ref{theorem:local well-posedness}, we assume that $\|\epsilon\|$ is sufficiently small such that the boundary input fulfils that
\begin{equation}\label{eq:bcinput1}
    \|h\|_{H^{m+1}(\Sigma)}\leq\sum\limits_{i=1}^N \epsilon_i\|h^{(i)}\|_{H^{m+1}(\Sigma)}<\frac{\epsilon_0}{2},
\end{equation}
where $\epsilon_0$ is the parameter introduced in Theorem \ref{theorem:local well-posedness}.

Let $u^{(i)}(t,x)$ denote the first-order linearization of $u$ in the sense that
\begin{equation}\label{eq:fol1}
u(t,x, \epsilon_i)=u(t,x,0)+u^{(i)}(t, x)\epsilon_i+o(\epsilon_i),\ \ i=1,2,\ldots, N, 
\end{equation}
{as $\epsilon_i\rightarrow 0$, where the generic constant in bounding the remainder term $o(\epsilon_i)$ depends on $\|u\|_{E^{m+1}}$}. For notational convenience, we write 
\begin{align*}
    u^{(i)}(t,x) &:= \frac{\partial}{\partial \epsilon_i} u(t,x) \Bigg|_{\epsilon_i=0} = \lim_{\epsilon_i \to 0} \frac{u(t,x,\epsilon_i) - u(t,x,0)}{\epsilon_i}, \quad i = 1, \ldots, N,
\end{align*}
and it is assumed that $u(t,x,0)=0$ in accordance with the third admissibility condition in Definition \ref{def:adm con}. Furthermore,
through direct computations, one can obtain that the first-order linearizations for the functions $F$ and $G$ are given by
\begin{equation}\label{eq:fol2}
\begin{split}
     \partial_{\epsilon_i} F(t,x,u)\big|_{\epsilon_i=0}=& \lim\limits_{\epsilon_i\to 0} \frac{F(t,x,u(t,x,\epsilon_i))-F(t,x,u(t,x,0))}{\epsilon_i}\\
    =& \frac{b_2(t,x)}{2!}\lim\limits_{\epsilon_i\to 0} \left[u(t,x,\epsilon_i)\frac{u(t,x,\epsilon_i)}{\epsilon_i}\right]
    =0,
\end{split}
\end{equation}
and
\begin{equation}\label{eq:fol3}
\begin{split}
     \partial_{\epsilon_i} G(t,x,u)\big|_{\epsilon_i=0}=& \lim\limits_{\epsilon_i\to 0} \frac{G(t,x,u(t,x,\epsilon_i))-G(t,x,u(t,x,0))}{\epsilon_i}\\
    =&\sum\limits_{k=1}^{\infty} \frac{c_k(t,x)}{k!}\lim\limits_{\epsilon_i\to 0} \frac{(u^k(t,x,\epsilon_i))}{\epsilon_i}
    =c_1(t,x)u^{(i)}.
\end{split}
\end{equation}
Here and also in what follows, it is emphasized that both \eqref{eq:fol2} and \eqref{eq:fol3} should be understood in a similar sense to that of \eqref{eq:fol1} in proper function spaces which should clear from the context. 

By plugging \eqref{eq:fol1}, \eqref{eq:fol2} and \eqref{eq:fol3} into \eqref{equ:h_expansion}, along with straightforward calculations, one can derive the following hyperbolic system: 
\begin{equation}
\label{eq:1st le}
\begin{cases}
    \partial_t^2 u^{(i)}=a(x)\Delta_{g} u^{(i)} + c_1(t,x) u^{(i)}, & \text{on } \mathcal{M}, \\
    u^{(i)}(t,x) = h^{(i)}, & \text{on } \Sigma, \\
    u^{(i)}(0,x) = 0, \quad \partial_t u^{(i)}(0,x) = 0, & \text{on } M,
\end{cases}
\end{equation}
for $i=1,\cdots,N$. The system \eqref{eq:1st le} is referred to as the first-order linearization of \eqref{equ:h_expansion}. It can be readily seen that the DtN maps associated with \eqref{eq:1st le} are given by
\begin{equation}\label{eq:1st dtn}
\begin{split}
    \Lambda^{(i)}(h):=&\partial_{\epsilon_i} \Lambda(h,\epsilon_i)\big|_{\epsilon_i=0}:=\lim\limits_{\epsilon_i\to 0} \frac{\Lambda(h,\epsilon_i)-\Lambda(h,0)}{\epsilon_i} \\
    =&\lim\limits_{\epsilon\to 0} \frac{\partial_{\nu} (u+F(t,x,u))\big|_{\Sigma}}{\epsilon_i}=\partial_{\nu} u^{(i)}\big|_{\Sigma}, \quad i=1,\cdots,N.
\end{split}
\end{equation}

We proceed to derive the second-order linearized systems by following similar calculations for the first-order ones through differentiating with respect to $\epsilon_{i_1}$ and $\epsilon_{i_2}$ for $1\leq i_1<i_2\leq N$; that is
\begin{equation}
\label{equ:Second_Order_linearization_equation}
    \left\{
    \begin{aligned}
        &\partial_t^2 u^{(i_1i_2)} = a(x)\Delta_{g} \left(u^{(i_1i_2)}+b_2(t,x)u^{(i_1)}u^{(i_2)}\right)+c_1(t,x) u^{(i_1i_2)}+c_2(x)u^{(i_1)}u^{(i_2)},\ &&\text{on}\ \mathcal{M},\\
        &u^{(i_1i_2)}(t,x)=0,\ &&\text{on}\ \Sigma,\\
        &u^{(i_1i_2)}(0,x)=0,\ \partial_t u^{(i_1i_2)}(0,x) = 0,&&\text{on}\ M,
    \end{aligned}
\right.
\end{equation}
where $u^{(i_1i_2)}=\partial_{\epsilon_{i_1}\epsilon_{i_2}}^2 u\big|_{\epsilon_{i_1},\epsilon_{i_2}=0}$. The corresponding second-order DtN maps are given by
\begin{eqnarray}
\label{equ:Second_Order_linearization_DtN}
    \Lambda^{(i_1i_2)}(h)&:=&\partial^2_{\epsilon_{i_1}\epsilon_{i_2}} \Lambda(h,\epsilon)\big|_{\epsilon_{i_1},\epsilon_{i_2}=0} \nonumber\\
   & =&\partial_{\nu} \left(u^{(i_1i_2)}+b_2(t,x)u^{(i_1)}u^{(i_2)}\right)\big|_{\Sigma},
    \quad 1\leq i_1<i_2\leq N.
\end{eqnarray}

Finally, for $k>2$, by differentiating with respect to $k$ parameters $\epsilon_{i_1}, \cdots, \epsilon_{i_k}$ for $1\leq i_1<\cdots<i_k\leq N$,
we can obtain the equations satisfied by the $k$th-order linearization $u^{(i_1\cdots i_k)}=\partial_{\epsilon_{i_1}\cdots\epsilon_{i_k}}^k u\big|_{\epsilon_{i_1},\cdots,\epsilon_{i_k}=0}$ of $u$:
\begin{equation*}
\label{equ:k-th_Order_linearization_equation}
    \left\{
    \begin{aligned}
        &\partial_t^2 u^{(i_1\cdots i_k)}-a(x)\Delta_{g} u^{(i_1\cdots i_k)}-c_1(t,x) u^{(i_1\cdots i_k)}\\&=\frac{1}{(k-1)!}\left[c_k(t,x)u^{(i_1)}\cdots u^{(i_k)}\right]
        +a(x)\Delta_g F^{(i_1\cdots i_k)}+G^{(i_1\cdots i_k)},\ &&\text{on}\ \mathcal{M},\\
        &u^{(i_1\cdots i_k)}(t,x)=0,\ &&\text{on}\ \Sigma,\\
        &u^{(i_1\cdots i_k)}(0,x)=0,\ \partial_t u^{(i_1\cdots i_k)}(0,x) = 0,&&\text{on}\ M,
    \end{aligned}
\right.
\end{equation*}
where $F^{(i_1\cdots i_k)}$ and $G^{(i_1\cdots i_k)}$ are functions depending on the lower-order linearizations of $u$ as well as the coefficients $b_2, c_1, \cdots, c_{k-1}$.
Moreover, the associated $k$th-order DtN maps are given by
\begin{equation}
    \label{equ:k-th_Order_linearization_DtN}
    \begin{aligned}
    \Lambda^{(i_1\cdots i_k)}:=&\partial_{\epsilon_{i_1}\cdots\epsilon_{i_k}}^k \Lambda(h,\epsilon_{i_1},\cdots,\epsilon_{i_k})\big|_{\epsilon_{i_1},\cdots,\epsilon_{i_k}=0}\\
    =&\partial_{\nu}\left(u^{(i_1\cdots i_k)}+F^{(i_1\cdots i_k)}\right)\big|_{\Sigma}.
\end{aligned}
\end{equation}

\subsection{Gaussian beams and their properties}
\label{sec:Gaussian_Beam_Solution}

In this subsection, we construct a special class of solution to the associate linear hyperbolic equations on Lorentian manifolds, which are known as the Gaussian beam solutions and will serve as the ``probing modes" for our inverse problems. 
They are required to satisfy the following equation:
\begin{equation}
\label{equ:Gaussian_beam_equation}
    \left\{
    \begin{aligned}
        &\partial_t^2 u=a(x)\Delta_g u+c_1(t,x) u\ &&\text{on}\ \mathcal{M},\\
        &u(0,x)=0,\ \partial_t u(0,x) = 0 &&\text{on}\ M.
    \end{aligned}
\right.
\end{equation}

In what follows, we begin with a brief discussion of some preliminary results. Then we construct the desired Gaussian beams. Finally, we provide several critical properties of these solutions for our subsequent use to the inverse problems. 


\subsubsection{Construction of Fermi Coordinates}\label{sec:Fermi}

Assuming $\gamma$ is a unit-speed geodesic passing through a point $(t_0,x_0)\in \mathcal{M}$, we next introduce the Fermi coordinates in a neighborhood of the null geodesic $\alpha(t)=(t,\gamma(t))$. Let $\alpha(t)$ join the two points $(t_{-},\gamma(t_{-}))$ and $(t_{+},\gamma(t_{+}))$, where $t_{-},t_{+}\in (0,T)$ and $\gamma(t_{-}),\gamma(t_{+})\in \partial M$. We can extend $\alpha(t)$ to $\mathcal{M}$ such that $\alpha(t)$ is defined on $[t_{-}-\delta,t_{+}+\delta]\subset(0,T)$ with $\delta$ being a small positive constant. As the formal construction in \cite{UZ21a}, for the new metric defined as $\widetilde{g} = -dt^2+a(x)g$, there exists a coordinate neighborhood $(U,\Phi)$ of $\gamma([t_{-},t_{+}])$, with the coordinates denoted by $(z_0:=s,z_1,z_2,z_3)$, satisfying:
\begin{enumerate}
\item $\Phi(U) = (t_{-}-\delta',t_{+}+\delta')\times B(0,\delta')$ where $B(0,\delta')$ denotes a ball in $\mathbb{R}^3$ with small radius $0<\delta'<\delta$;
\item $\Phi(\gamma(s)) = (s,0,0,0)$.
\end{enumerate}
Furthermore, the corresponding metric can be rewritten as
\[
    \widetilde{g}\left|_{\alpha}\right.=2ds\otimes dz_1+\sum\limits_{i=2
    }^3 dz_i\otimes dz_i,
\]
and possesses the following properties
\[
\frac{\partial \widetilde{g}_{i_2i_3}}{\partial z_{i_1}}\left|_{\alpha}\right.=0,\, \quad 0\leq i_1,i_2,i_3\leq 3.
\]


\subsubsection{Construction of desired Gaussian beams}\label{sec:existence of GM}

Following a similar spirit to the study in \cite{UZ21b} and utilizing the Fermi coordinate system, we construct the Gaussian beam solutions
that are concentrated along a given null geodesic $\alpha:I=(t_{-}-\delta,t_{+}+\delta)\to \mathcal{M}$.
By the construction of Fermi Coordinates in the previous subsection, we can obtain a neighborhood of this geodesic:
\begin{equation}
	\mathcal{V}=\left\{x=(s,z^{\prime}) \in \mathcal{M}~|~s \in I,\,| z^{\prime} \mid:=\sqrt{\left|z_1\right|^2+\left|z_2\right|^2+\left|z_3\right|^2}<\delta^{\prime}\right\},
\end{equation}
with $\delta'$ defined in Section \,\ref{sec:Fermi}.
In the neighborhood $\mathcal{V}$, we assume that the solution has a Wentzel-Kramers-Brillouin (WKB) ansatz as
\begin{equation}\label{equ:utau}
	u_\tau\left(s, z^{\prime}\right)=e^{\mathrm{i} \tau \phi\left(s, z^{\prime}\right)} A\left(s, z^{\prime}\right),\ \ \mathrm{i}:=\sqrt{-1},
\end{equation}
where the complex-valued phase function $\phi \in C^{k}(\mathcal{V})$ and the amplitude function $A \in C_c^{k}(\mathcal{V})$ are to be determined for a positive integer $k$.
Substituting \eqref{equ:utau} into
\begin{equation}
	\label{operator_L_eqn}
	\mathcal{L}_{a,c_1}u_{\tau}:=(-\partial_t^2+a(x)\Delta_g+c_1(t,x))u_{\tau}=0,
\end{equation}
we have
\begin{equation*}
	\mathcal{L}_{a,c_1}\left(e^{\mathrm{i} \tau \phi} A\right)=e^{\mathrm{i} \tau \phi}\left(-\tau^2(\mathcal{S} \phi) A+\mathrm{i} \tau \mathcal{T} A+\mathcal{L}_{a,c_1} A\right)=0,
\end{equation*}
where
\begin{equation*}
	\mathcal{S} \phi:=\langle d \phi, d \phi\rangle_{\widetilde{g}}, \quad \mathcal{T} A:=2\langle d \phi, d A\rangle_{\widetilde{g}}+A \left(\triangle_{\widetilde{g}} \phi\right).
\end{equation*}

To ensure that the constructed solution exhibits a vanishing of order $N$ along the null geodesic $\alpha$, we further assume the following ansatz for the phase and amplitude functions:
\begin{equation}
	\phi\left(s, z^{\prime}\right)=\sum_{k=0}^N \phi_k\left(s, z^{\prime}\right) \quad \text { and } \quad A\left(s, z^{\prime}\right)=\chi\left(\frac{\left|z^{\prime}\right|}{\delta^{\prime}}\right) \sum_{k=0}^N \tau^{-k} A_k\left(s, z^{\prime}\right) \text {, }
\end{equation}
where $\phi_k$ is a complex-valued $k$th degree homogeneous polynomial with respect to the variables $z_1, z_2, z_3$, and $\chi(t)$ is a non-negative smooth cutoff function with a compact support such that $\chi(t)=1$ for $|t|\leq\frac{1}{4}$ and $\chi(t)=0$ for $|t|\geq\frac{1}{2}$.

The Gaussian beam solution exhibiting a vanishing of order $N$ along the null geodesic $\alpha$ is equivalent to the requirements as described below.
First, we require the phase function to satisfy the following set of equations:
\begin{equation}
	\label{operator_S_eqn}
	\frac{\partial^\Theta}{\partial z^\Theta}\left(\mathcal{S} \phi\right)(s, 0, \ldots, 0)=0, \quad \forall s \in I,
\end{equation}
where $\Theta=(0,\Theta_1,\Theta_2,\Theta_3)$ satisfies $|\Theta|=\Theta_1+\Theta_2+\Theta_3\leq N$. In particular, we refer to the case where $\Theta=0$ as the eikonal equation.
Second, we also require that the leading term of the amplitude function, namely $a_0$, satisfies the following transport-type equation:
\begin{equation}
	\label{operator_T_eqn}
	\frac{\partial^\Theta}{\partial z^\Theta}\left(\mathcal{T} A_0\right)(s, 0, \ldots, 0)=0, \quad \forall s \in I.
\end{equation}
Finally, we need the adjacent terms of the amplitude to satisfy the following recursive equation
\begin{equation}
	\label{operator_T_and_L_eqn}
	\frac{\partial^\Theta}{\partial z^\Theta}\left(i \mathcal{T} A_k\right)(s, 0, \ldots, 0)=-\frac{\partial^\Theta}{\partial z^\Theta}\left(\mathcal{L}_{a,c_1} A_{k-1}\right)(s, 0, \ldots, 0), \quad \forall s \in I,
\end{equation}
where $k=1,2,\cdots,N$.

To summarize, the solution obtained by satisfying equations \eqref{operator_L_eqn}, \eqref{operator_S_eqn}, \eqref{operator_T_eqn}, and \eqref{operator_T_and_L_eqn} for the phase and amplitude is referred to as an approximate Gaussian beam solution of order $N$.

As per \cite{FO22}, the phase function of such solutions exhibits the following two properties:
\begin{equation}
	\label{property_of_phase_function}
	\left.\Im(\phi\left(s, z^{\prime}\right))\right|_\alpha=0, \quad \Im(\phi\left(s, z^{\prime}\right))\geq C|z^{\prime}|^2, \quad\, \forall z=(s,z')\in\mathcal{V},
\end{equation}
where $\Im(v)$ is the image part of $v$ for any $v\in \mathbb{C}$. Moreover, we have the following estimates
\begin{equation}
	\left\|\mathcal{L}_{a,c_1} u_\tau\right\|_{H^k((0, T) \times M)} \lesssim \tau^{-K}, \quad\left\|u_\tau\right\|_{\mathcal{C}((0, T) \times M)} \lesssim 1,
\end{equation}
where $K=\frac{N+1-k}{2}-1$. Here and also in what follows, $``\mathscr{A}\lesssim\mathscr{B}"$ signifies that $\mathscr{A}\leq C\mathscr{B}$ for a generic constant $C$. These estimates can be directly computed from \eqref{operator_L_eqn}, \eqref{operator_S_eqn}, \eqref{operator_T_eqn}, \eqref{operator_T_and_L_eqn}, and the second property of the phase function is given in (\ref{property_of_phase_function}).

Let us show how the phase function and amplitude function can be constructed through \eqref{operator_S_eqn}, \eqref{operator_T_eqn}, and \eqref{operator_T_and_L_eqn}. First, we construct the $\phi_0$ component of the phase function $\phi$. We set $\Theta=(0,0,0,0)$ and expand \eqref{operator_S_eqn} to obtain the following equation:
\begin{equation}
	\sum_{i_1, i_2=0}^3 \widetilde{g}^{i_1 i_2} \frac{\partial \phi}{\partial z^{i_1}} \frac{\partial \phi}{\partial z^{i_2}}=0.
\end{equation}
On the null geodesic, the metric is $\left.\widetilde{g}\right|_{\alpha}=2 ds \otimes dz_1+\sum_{i=2}^3 dz_i \otimes dz_i$. Therefore, the above equation can be rewritten as
\begin{equation}
	\label{phase_eqn_order_first}
	2 \partial_0 \phi \partial_1 \phi+\sum_{i=2}^3\left(\partial_i \phi\right)^2=0.
\end{equation}
Utilizing $\frac{\partial \widetilde{g}_{i_2i_3}}{\partial z_{i_1}}\left|_{\alpha}\right.=0$, for all $0\leq i_1,i_2,i_3\leq 3$, we can transform the equation for $|\Theta|=1$ into
\begin{equation}
	\label{phase_eqn_order_second}
	\sum_{i_2, i_3=0}^n \widetilde{g}^{i_2 i_3} \frac{\partial^2 \phi}{\partial z^{i_1}\partial z^{i_2}} \frac{\partial \phi}{\partial z^{i_3}}=0,
\end{equation}
where $i_1=1,2,3$ corresponds to the three cases where $\Theta_1=1$, $\Theta_2=1$, and $\Theta_3=1$, respectively. For \eqref{phase_eqn_order_first} and \eqref{phase_eqn_order_second}, we only need to take $\phi_0\left(s, z^{\prime}\right)=0$ and $\phi_1\left(s, z^{\prime}\right)=z_1$ to satisfy the equations.
For $|\Theta|=2$, we further assume $\phi_2\left(s, z^{\prime}\right):=\sum_{1 \leqslant i_1, i_2 \leqslant 3} H_{i_1i_2}(s) z^{i_1} z^{i_2}$, where $H_{i_1i_2}=H_{i_2i_1}$ is a complex-valued symmetric matrix to be determined. Through the properties of the metric and equations \eqref{phase_eqn_order_first} and \eqref{phase_eqn_order_second}, we obtain a Riccati equation satisfied by $H$ as follows
\begin{equation}\label{equ:H}
	\frac{d}{d s} H+H C H+D=0, \quad \forall s \in I,
\end{equation}
where $C$ and $D$ are given by
\begin{equation}\label{equ:CD}
C= \left(
\begin{array}{ccc}
0 & 0 &  0\\
0 & 2 &  0\\				
0 & 0 &  2\\
\end{array}
\right),\quad  D_{i_1i_2}= \frac{1}{4} \partial_{i_1i_2}^2\widetilde{g}^{11}\left|_{\alpha}\right.,\, i_1,i_2=1,2,3.
\end{equation}
From \cite{FO22}, we can ascertain the existence of a solution to this equation in the next subsection. For the case of $3\leq |\Theta|\leq N$, through similar calculations, we can derive a linear first-order ODE system whose right-hand side only contains low-order terms of $\phi_i(i<|\Theta|)$. By solving this system of equations, we can obtain the corresponding higher-order terms of $\phi_{|\Theta|}$.

Next, we briefly outline the construction of the amplitude function $A$. By similar calculations as the above, we obtain the equation satisfied by the leading term $A_0$:
\begin{equation}
	\label{amplitude_eqn_order_first}
	2 \sum_{i_1, i_2=0}^3 \widetilde{g}^{i_1 i_2} \frac{\partial \phi}{\partial z^{i_1}} \frac{\partial \phi}{\partial z^{i_2}}+\left(\triangle_{\tilde{g}} \phi\right) A_0=0, \quad \forall s \in I.
\end{equation}
Based on the earlier calculations for constructing the phase function $\phi$, we can simplify the above equation to
\begin{equation}
	2 \frac{d}{d s} A_{0}+\operatorname{Tr}(C H) A_{0}=0, \quad \forall s \in I,
\end{equation}
which results in the leading term function $A_0$.
For higher-order amplitude functions $A_k(k\geq 1)$, we can derive them through the recursive equation \eqref{operator_T_and_L_eqn} and
$\{A_0,\cdots,A_{k-1}, \phi_0,\cdots,\phi_N\}$. These can be simplified into the corresponding first-order linear ODEs which are solvable by the classical ODE theory. Since the subsequent equations involved are all first-order linear ODEs, we can systematically obtain the corresponding amplitude functions term by term.

In conclusion, we can construct the desired Gaussian beam solutions of order $N$.


\subsubsection{Properties of the Gaussian beams}

In this subsection, we outline some properties of the Gaussian beam solutions, which are instrumental in the subsequent proofs of the main theorems for our inverse problems.

First, we discuss the existence and properties of $H$ as defined in (\ref{equ:H}).
\begin{Lemma}\label{lem:H} (\cite{FIKO21,FO22})
Let $H_0\in \mathbb{C}^{3\times 3}$ be a symmetric matrix with $\Im(H_0)>0$ and $C$ and $D$ are defined by (\ref{equ:CD}).
Then there exists a unique symmetric matrix $H\in C^3(I;\mathbb{C}^{3\times 3})$ satisfying a Riccati-type equation as
\begin{align*}
	\frac{d}{ds}H+HCH+D=0,\,\,\, s\in I,
\end{align*}
with the initial value $H_0$. Moreover, $H$ satisfies $\Im(H(s))>0$ and we can express $H$ as $H(s) = Z(s)Y^{-1}(s)$, where $Z\in C^3(I;\mathbb{C}^{3\times 3})$ and $Y\in C^4(I;\mathbb{C}^{3\times 3})$ solve the following ODEs
\begin{equation}\label{equ:Y and Z}
	\left\{
	\begin{aligned}
		&\frac{d}{ds}Y = CZ,\,\,\,\ &&Y(t_{-}-\delta')=I,\\
		&\frac{d}{ds}Z = -DY,\,\,\,\ &&Z(t_{-}-\delta')=H_0,
	\end{aligned}
	\right.
\end{equation}
and $I$ denotes the identity matrix.
\end{Lemma}
	
	With the definition of $Y$ in Lemma\,\ref{lem:H}, we ascertain that
	\begin{align}
		\label{equ:amplitude_form}
		A_0(s) = \det(Y(s))^{-\frac{1}{2}}.
	\end{align}
is the solution of equation \eqref{operator_T_eqn}.

Second, we will need the following lemma which summarizes the properties of $Y$.
\begin{Lemma}\label{lem:Y}
(\cite{AUZ22,FIKO21})
Let the matrix $Y$ in Lemma~\ref{lem:H} be non-degenerate for any $s$. Then $Y$ satisfies the following second-order ODE:
\begin{equation}
	\frac{d^2}{d s^2} Y=-C D Y, \quad Y\left(t_{-}-\delta^{\prime}\right)=I, \quad \dot{Y}\left(t_{-}-\delta^{\prime}\right)=C H_0.
\end{equation}
Meanwhile, its determinant satisfies the following algebraic equation:
\begin{equation}
	\operatorname{det}(\Im H(s))|\operatorname{det}(Y(s))|^2=\operatorname{det}\left(\Im H_0\right).
\end{equation}

It follows from \cite{FIKO21} that $\left.\partial^2_{1i_2} \widetilde{g}^{11}\right|_{\alpha}=0$, $i_2=1,2,3$ and $\left.\partial^2_{i_1i_2} \widetilde{g}^{11}\right|_{\alpha}=-\left.R_{1 i_1 1 i_2}\right|_{\gamma}$, $i_1,i_2=1,2,3$,
where $R=(R_{i_1 i_2 i_3 i_4})_{i_1i_2i_3i_4=1}^{3}$ is the Riemann curvature tensor on the Riemann manifold ($M,g$).
Specifically, if we set $Y_{11}=1$ and $Y_{1i_2}=Y_{i_11}=0$ ($i_1, i_2=2,3$), and let $\widetilde{Y}=\left(Y_{i_1}^{i_2}\right)_{i_1, i_2=2}^{3}$, then
\begin{equation}
\det \left(Y\right)=\det (\widetilde{Y}).
\end{equation}
\end{Lemma}


Third, we provide the estimates of the remainder term. Through the construction in Section\,\ref{sec:existence of GM},
we ascertain that the Gaussian beam solution $u_\tau$ concentrates on the null geodesic $\alpha$. Let
\begin{equation}
    \label{equ:Gaussian_beam_u,v}
    \begin{aligned}
            &u(t,x) =u_\tau (t, x)+R_{u,\tau}(t,x),\\
    \end{aligned}
\end{equation}
where $R_{u,\tau}$ is the remainder term satisfying
\begin{equation*}
\left\{
\begin{aligned}
        &\mathcal{L}_{a,c_1}R_{u,\tau} = F_{u,\tau},&&\ \text{on }\mathcal{M},\\
        &R_{u,\tau}=0,&&\ \text{on } \Sigma,\\
    &R_{u,\tau}(0,x)=0,\ \partial_t R_{u,\tau}(0,x)=0, &&\ \text{on } M,
\end{aligned}\right.
\end{equation*}
with $F_{u,\tau}=-e^{\mathrm{i}\tau\phi}\left[\tau^2 A(S\phi)-\mathrm{i}\tau\mathcal{T}A+\mathcal{L}_{a,c_1}A\right]$. On one hand, using the $H^{m}$ estimate of hyperbolic equations in \cite{P87} for $m\geq 1$, we have
\begin{align*}
    \|R_{u,\tau}\|_{H^{m}(\mathcal{M})}\leq C\|F_{u,\tau}\|_{H^{m-1}(\mathcal{M})}.
\end{align*}
Moreover, it is straightforward to verify that, for $\Im \phi > 0$, we have
\begin{align*}
    |e^{\mathrm{i}\tau\phi}|\leq e^{-D\tau |z'|^2}
\end{align*}
with $D>0$ only depending on the underlying geometry. It follows from \cite{FIKO21} that
\begin{align*}
    &\|e^{\mathrm{i}\tau \phi}S\phi\|_{L_2}^2 \lesssim \int_{\mathcal{V}} |z'|^4e^{-D\tau |z'|^2}\chi^2\left(\frac{|z'|}{\delta} \right) dz=O(\tau^{-\frac{5}{2}}),\\
    &\|e^{\mathrm{i}\tau\phi}\mathcal{T} A\|_{L_2}^2\lesssim \int_{\mathcal{V}} |z'|^2e^{-D\tau |z'|^2}\chi^2\left(\frac{|z'|}{\delta} \right)dz = O(\tau^{-\frac{3}{2}}),\\
    &\|e^{\mathrm{i}\tau\phi}\|_{L_2}^2\lesssim \int_{\mathcal{V}} e^{-D\tau |z'|^2}\chi^2\left(\frac{|z'|}{\delta} \right)dz = O(\tau^{-\frac{1}{2}}).
\end{align*}
Combining the estimates above, we can obtain
\begin{align}
\label{equ:Ru_H1_property}
    \|R_{u,\tau}\|_{H^m(\mathcal{M})}\leq C\tau^{m-1/2},
\end{align}
which will play a crucial role in the subsequent proofs.

Finally, we conclude this subsection with a remark on the backward problem related to (\ref{equ:Gaussian_beam_equation}).
\begin{Remark}
For the  backward problem associated with (\ref{equ:Gaussian_beam_equation}), where the initial condition in \eqref{equ:Gaussian_beam_equation} is altered by
\begin{align*}
	u(T,x) = \partial_{t} u(T,x)=0,\,\,\, x\in M,
\end{align*}
all the properties regarding $H$, $Y$, and the estimation of the remainder still hold.
\end{Remark}


\subsection{Proof of Theorem~\ref{Theo:IP11}}

In this section, we provide the proof of Theorem~\ref{Theo:IP11}, which adopts and develops the method of ray transforms from \cite{FIKO21,FO20}.

Considering an interval $I$, we set $\gamma$ as the maximal geodesic in $M$ without any conjugate points, $\alpha$ as the maximal geodesic in $\mathcal{M}$, and $\widetilde{Y}$ as defined in Lemma\,\ref{lem:Y}. We then explore the following two ray transformations:
\begin{eqnarray}\label{eqn: trans Jm}
\mathcal{J}^{(m)}_{\widetilde{Y}} f_1&:=& \int_{t\in I} f_1(\gamma(t))\left|\det \widetilde{Y}\right|^{-1}(\det(\widetilde{Y}(t)))^{-m/2}\mathrm{d}t,\, m\ge 1,\\
\mathcal{J}^{(2)} f_2&:=& \int_{t\in I} f_2(\alpha(t))\mathrm{d}t,
\end{eqnarray}
where $f_1\in C(M)$, and $f_2\in C(\mathcal{M})$ with supp$f_2\in \mathcal{E}$ as defined in \eqref{equ:def_E}.
For the former one, we have the following properties without $\left|\det \widetilde{Y}\right|^{-1} $ (cf.\cite{UZ21b}):
\begin{align*}
	\forall\, f_1\in C(M),\, \mathcal{J}_{\widetilde{Y}}^{(m)} f_1=0,  \, m=1,  \quad \Longrightarrow \quad  f_1(\alpha(t))=0,\, \forall\, t\in I,
\end{align*}
and for the latter one, we have (cf.\cite{FIKO21,FO20}):
\begin{align*}
 \forall\, f_2\in C(\mathcal{M}),\, \mathcal{J}^{(2)} f_2=0, \quad \Longrightarrow \quad  f_2(\alpha(t))=0,\, \forall\, t\in I.
\end{align*}
Here, we need to provide a technical proof for the former one with $m\geq 3$.
\begin{Lemma}
\label{ray_transform}
Let $(M,g)$ be a three-dimensional compact smooth Riemannian manifold with boundary. Let $\gamma$ be a maximal geodesic in $M$ without any conjugate points, and $m\geq 1$. Then, the following result holds:
\begin{equation*}
 \forall\, f_1\in C({M}), \, \mathcal{J}^{(m)}_{\widetilde{Y}} f_1=0 \ \  \Longrightarrow \quad  f_1(\gamma(t))=0,\,  \forall\, t\in I,
\end{equation*}
which implies that $\mathcal{J}^{(m)}_{\widetilde{Y}}$ is injective.
\end{Lemma}
\begin{proof}
 To prove that $\mathcal{J}^{(m)}_{\widetilde{Y}}$ is injective, it is sufficient to show that
$f_1(p)=0$ for all $p\in \gamma$.

Without loss of generality, we can choose any point $p\in\gamma$ and reparametrize it so that $I=[-t_1,t_1]$ and $p=\gamma(0)$.
At the point $p$, we take an orthonormal basis $\{v_1, v_2\}$ of the tangent space $\dot{\gamma}^{\perp}(0)$.
Recall that the complex Jacobi equation for $\widetilde{Y}\in C^4(I;\mathbb{C}^{2\times 2})$ is given by
\[
\begin{split}
& \frac{d^2}{d t^2} \widetilde{Y}(t)-K(t)\widetilde{Y}(t)=0,\\
&\widetilde{Y}(0)=(0_{2\times 1}, 0_{2\times 1}),
\quad \frac{d}{d t} {\widetilde{Y}}(0)=(v_1,v_2),
\end{split}
\]
where $K=K_{i_2}^{i_3} \frac{\partial}{\partial x^{i_3}} \otimes d x^{i_2}$ with $K_{i_2}^{i_3}=\widetilde{g}^{{i_3} {i_1}} \mathcal{R}_{{i_1} {i_2}}$, $i_1$, $i_2$, $i_3=2,3$. Here, $\mathcal{R}=(\mathcal{R}_{{i_1} {i_2}})_{i_1, i_2=2}^3$ is the Ricci tensor.
For $0<\eta < t_1$, we let $\widetilde{Y}^{\eta}$ be the solution of the perturbation equation; that is
\begin{eqnarray*}
&\displaystyle \frac{d^2}{d t^2} \widetilde{Y}^{\eta}(t)-K(t)\widetilde{Y}^{\eta}(t)=0,\\
&\widetilde{Y}^{\eta}(0)=-i\eta (v_1,v_2),
\quad
\displaystyle \frac{d}{d t} {\widetilde{Y}^{\eta}}(0)=(v_1,v_2).
\end{eqnarray*}

Let $\mathcal{J}^{(m)}_{\widetilde{Y}^{\eta}}$ be the transformation given by (\ref{eqn: trans Jm}) with  $\widetilde{Y}$ replaced by $\widetilde{Y}^{\eta}$. Choose $\tilde{\eta}$ such that $0<\eta<\tilde{\eta}<t_1$ and define
\begin{align*}
	&\mathcal{J}_{\widetilde{Y}^{\eta}}^{(m)}f_1 =\mathcal{I}_1+\mathcal{I}_2,
\end{align*}
where
\begin{eqnarray}\label{equ:I1}
	&\mathcal{I}_1 = \int_{-\widetilde{\eta}}^{\widetilde{\eta}} f_1(\gamma(t))\left|\det \widetilde{Y}^{\eta}(t)\right|^{-1}\left(\det\left(\widetilde{Y}^{\eta}(t)\right)\right)^{-m/2}\mathrm{d}t,\\
\label{equ:I2}
	&\mathcal{I}_2 = \int_{[-t_1,t_1]\setminus [-\widetilde{\eta},\widetilde{\eta}]} f_1(\gamma(t))\left|\det \widetilde{Y}^{\eta}(t)\right|^{-1}\left(\det\left(\widetilde{Y}^{\eta}(t)\right)\right)^{-m/2}\mathrm{d}t.
\end{eqnarray}
To estimate the two terms in \eqref{equ:I1} and \eqref{equ:I2}, we make use of the following estimates from \cite{FO20}:
\begin{eqnarray}
	\label{equ:det|Y|}
	&\left|1-\left|\det\left(\widetilde{Y}^{\eta}(t)\right)\right|^{-1}|t-\mathrm{i}\eta|^{-2}\right|\leq C_1\widetilde{\eta},\\
\label{equ:detY}	
& \det\left(\widetilde{Y}^{\eta}(t)\right) = (t-\mathrm{i}\eta)^2(1+r^{\eta}(t))\,\,\, \text{with}\,\,\, |r^{\eta}(t)|\leq C_2|t|, \quad \forall t\in (-\widetilde{\eta},\widetilde{\eta}),\\
& C_3^{-1} <\left|\det \widetilde{Y}^{\eta}(t)\right|^{-1}< C_3
\end{eqnarray}
where $C_1$, $C_2$ and $C_3$ are positive constants independent of $\eta$ and $\widetilde{\eta}$. Therefore, we can choose
$\widetilde{\eta}$ to be small enough such that the residue  $|r^{\eta}(t)|<\frac{1}{2}$.

We first consider $\mathcal{I}_2$.
Let $X_1$ be the solution of complex Jacobi equation with $X_1(0)=0_{2\times 2}$ and $\dot{X_1}(0)=I_{2\times 2}$, and let $X_2$
be the solution of complex Jacobi equation with $X_2(0)=I_{2\times 2}$ and $\dot{X_2}(0)=0_{2\times 2}$. Then $\widetilde{Y}^{\eta}=X_1-\mathrm{i}\eta X_2$.

Since $X_1$ is non-degenerated on $[-t_1,t_1]\setminus[-\widetilde{\eta},\widetilde{\eta}]$, we have
$|\det\left(X_1(t)\right)|>C_4(\widetilde{\eta})$, where $C_4$ is a positive constant dependent on $\widetilde{\eta}$.
On the other hand, we can choose $\eta$ to be small enough such that $|\det\left(\widetilde{Y}^{\eta}(t)\right)-\det\left(X_1(t)\right)|<C_5\eta$.
These facts, together with the definition of $\mathcal{I}_2$ and the triangle inequality, imply
\begin{align}
	\label{estimate_I_2}
	|\mathcal{I}_2|<C_{6}(\widetilde{\eta}).
\end{align}
where $C_5$ is a positive constant dependent on $\widetilde{\eta}$.

Next, we consider $\mathcal{I}_1$, which can be split into three parts as follows:
\begin{align*}
	&\mathcal{I}_1 = \mathcal{I}_{1,1}+\mathcal{I}_{1,2}+I_{1,3},
\end{align*}
where
\begin{align*}
	&\mathcal{I}_{1,1} = f_1(\gamma(0))\int_{-\widetilde{\eta}}^{\widetilde{\eta}} |t-i\eta|^{-2}\left|\det \widetilde{Y}^{\eta}(t)\right|^{-1}\left(\det\left(\widetilde{Y}^{\eta}(t)\right)\right)^{-m/2}\mathrm{d}t,\\
	&\mathcal{I}_{1,2} = \int_{-\widetilde{\eta}}^{\widetilde{\eta}} \left(f_1(\gamma(t))-f_1(\gamma(0))\right)|t-i\eta|^{-2}\left|\det \widetilde{Y}^{\eta}(t)\right|^{-1}\left(\det\left(\widetilde{Y}^{\eta}(t)\right)\right)^{-m/2}\mathrm{d}t,\\
	&\mathcal{I}_{1,3} =\int_{-\widetilde{\eta}}^{\widetilde{\eta}}f_1(\gamma(t))|t-i\eta|^{-2}
\left(-1+|t-i\eta|^{2}\det\left(\left(\widetilde{Y}^{\eta}(t)\right)\right)^{-1}\right)\left|\det \widetilde{Y}^{\eta}(t)\right|^{-1}\left(\det\left(\widetilde{Y}^{\eta}(t)\right)\right)^{-m/2}\mathrm{d}t.
\end{align*}

For the term $\mathcal{I}_{1,2}$, it follows from \eqref{equ:detY} and the continuity of $f_1$ that
\begin{equation}
		\label{estimate_I_1_2}
	\begin{aligned}
			|\mathcal{I}_{1,2}|&\leq C_7\omega_{f_1}\left(\widetilde{\eta}\right)\left(\int_{-\widetilde{\eta}}^{\widetilde{\eta}}|t-i\eta|^{-2}\left(\det\left(\widetilde{Y}^{\eta}(t)\right)\right)^{-m/2}\mathrm{d}t \right)\\
		&\leq C_7\omega_{f_1}\left(\widetilde{\eta}\right)\int_{-\widetilde{\eta}}^{\widetilde{\eta}}\left(\frac{1}{(t^2+\eta^2)^{1+m/2}}\frac{1}{|1+r^{\eta}(t)|^{m/2}}\right)\mathrm{d}t\\
		&\leq 2^{m/2}C_7\omega_{f_1}\left(\widetilde{\eta}\right)\int_{-\widetilde{\eta}}^{\widetilde{\eta}}\left(\frac{1}{(t^2+\eta^2)^{1+m/2}}\right)\mathrm{d}t,
	\end{aligned}
\end{equation}
where $\omega_{f_1}\left(\widetilde{\eta}\right)$ is the modulus of continuity of $f_1$ on $[-\widetilde{\eta},\widetilde{\eta}]$, and $C_6$ is a positive constant independent of $\eta$ and $\widetilde{\eta}$.

For the term $\mathcal{I}_{1,3}$, it follows from \eqref{equ:det|Y|}, the triangle inequality, and the continuity of $f_1$ that
\begin{equation}
	\label{estimate_I_1_3}
	|\mathcal{I}_{1,3}|\leq 2^{m/2}C_8\widetilde{\eta}E,
\end{equation}
where $E=\int_{-\widetilde{\eta}}^{\widetilde{\eta}}\left(\frac{1}{(t^2+\eta^2)^{1+m/2}}\right)\mathrm{d}t$ and $C_8$ is a positive constant independent of $\eta$ and $\widetilde{\eta}$.

Similarly, for the term $\mathcal{I}_{1,1}$, it follows from the boundedness of the residue $r^{\eta}$ that
\begin{equation}
	\label{estimate_I_1_1}
	|\mathcal{I}_{1,1}|\ge C_9|f_1(\gamma(0))|E,
\end{equation}
where $C_9$ is the positive constants independent of $\eta$ and $\widetilde{\eta}$.

On the other hand, through direct calculations, we have
\begin{equation}
	\label{equ:estimate_2_integral}
	\begin{aligned}
		&\int_{-t_1}^{t_1}\left(\frac{1}{(t^2+\eta^2)^{1+m/2}}\right)\mathrm{d}t\ge C_{10}\eta^{-1-m},\\
		&\int_{[-t_1,t_1]\setminus [-\widetilde{\eta},\widetilde{\eta}]}\left(\frac{1}{(t^2+\eta^2)^{1+m/2}}\right)\mathrm{d}t\leq C_{11}\widetilde{\eta}^{-1-m},
	\end{aligned}
\end{equation}
where $C_{10}$ and $C_{11}$ are positive constants independent of $\eta$ and $\widetilde{\eta}$. Then, for $\eta$ small enough, we have
\begin{equation}
\label{dividing_thing}
\frac{1}{E} \leq\frac{1}{\frac{C_{10}}{\eta^{m+1}}-\frac{C_{11}}{\widetilde{\eta}^{m+1}}}.
\end{equation}

Since $\widetilde{Y}^{\eta}$ is the perturbation of $\widetilde{Y}$, there is a constant $C_{12}$ such that $|\mathcal{J}_{\widetilde{Y}^\eta}^{(m)} f_1-\mathcal{J}_{\widetilde{Y}}^{(m)} f_1|<C_{12}(\eta,f_1)$ with $\lim\limits_{\eta\rightarrow0}C_{12}(\eta,f_1)=0$. Together with $\mathcal{J}_{\widetilde{Y}}^{(m)}f_1=0$, we obtain
\begin{equation}
\label{estimate_J_Y_eta}
|\mathcal{J}_{\widetilde{Y}^\eta}^{(m)} f_1|<|\mathcal{J}_{\widetilde{Y}^\eta}^{(m)} f_1-\mathcal{J}_{\widetilde{Y}}^{(m)} f_1|+|\mathcal{J}_{\widetilde{Y}}^{(m)} f_1|<C_{10}(\eta,f_1).
\end{equation}

Finally, combining the estimates \eqref{estimate_I_1_1}, \eqref{estimate_I_1_2}, \eqref{estimate_I_1_3}, \eqref{estimate_I_2}, \eqref{estimate_J_Y_eta}, and \eqref{dividing_thing}, we obtain
\begin{equation*}
	\begin{aligned}
		|f_1(\gamma(0))|&\leq C_{13}\frac{|\mathcal{I}_{1,1}|}{E}\\
		&\leq C_{13}\frac{|\mathcal{I}_{1,2}|+|\mathcal{I}_{1,3}|+|\mathcal{I}_{2}|+|\mathcal{J}_{\widetilde{Y}^\eta}^{(m)} f_1|}{E}\\
		&\leq C_{13}\left(\widetilde{\eta}+\omega_{f_1}(\widetilde{\eta})+\frac{|\mathcal{I}_{2}|+C_{12}}{E}\right).
	\end{aligned}
\end{equation*}
where $C_{13}$ is a positive constant independent of $\eta$ and $\widetilde{\eta}$. Letting $\eta\rightarrow0$, we have
\begin{equation}
	\begin{aligned}
		|f_1(\gamma(0))|\leq C_{14}(\widetilde{\eta}+\omega_{f_1}(\widetilde{\eta})).
	\end{aligned}
\end{equation}
Setting $\widetilde{\eta}\rightarrow0$, we finally obtain
\begin{align*}
|f_1(\gamma(0))|=0,
\end{align*}
which readily implies the conclusion.

The proof is complete. 
\end{proof}

We are now at the stage of proving the proof of Theorem \ref{Theo:IP11}. As discussed earlier, we assume $n=2$ in \eqref{eq:mam1}, which means that $F$ is a quadratic monomial. Here, it is emphasized again that the case with a general $n$ can be proved in a completely similar manner. We need to prove the uniqueness of $a,b_2$ and $c_{k}$ for $k\geq 1$, i.e. when $\Lambda_{a_1,F_1,G_1}=\Lambda_{a_2,F_2,G_2}$, we need to show that one must have $a_1=a_2, b_{2,1}=b_{2,2}$ and $c_{k,1}=c_{k,2}$ for $k\geq 1$.
Here, $a_j,b_{k,j}$ and $c_{k,j}$ are the coefficients according to $\Lambda_{a_j,F_j,G_j}$, $j=1,2$.


First, we show the uniqueness of $c_1$ in (\ref{eq:g ck}). Let $u_{j},j=1,2$ be the solutions of
\begin{equation}
	\label{equ:Quasilinear_equation_j}
	\left\{
	\begin{aligned}
		&\partial^2_t u_j = a_j\Delta_g (u_j+F_j(x,u_j))+G_j(x,u_j),\ &&\text{on}\ \mathcal{M},\\
		&u_j(t,x)=\sum\limits_{i=1}^N\epsilon_ih^{(i)},\ &&\text{on}\ \Sigma,\\
		&u_j(0,x)=0,\ \partial_t u_j(0,x) = 0,&&\text{on}\ M,
	\end{aligned}
	\right.
\end{equation}
according to $\Lambda_{a_j,F_j,G_j}$, where
\begin{align*}
			F_j(x,y) = b_{2,j}y^2,\,\,\, G_j(x,y) = \sum\limits_{k=1}^{\infty} c_{k,j}(x)y^k.
\end{align*}
Let $N=1$. The first-order linearizations of $u_{j}$ $(j=1,2)$ are denoted by $u_j^{(1)}=\partial_{\epsilon_1}u_j\big|_{\epsilon_1=0}$, which satisfy
\begin{equation*}
	\left\{
	\begin{aligned}
		&\partial_t^2 u_j^{(1)}=a_j(x)\Delta_g u_j^{(1)}+c_{1,j}(x) u_j^{(1)},\ &&\text{on}\ \mathcal{M},\\
		&u_j^{(1)}(x)=h^{(1)},\ &&\text{on}\ \Sigma,\\
		&u_j^{(1)}(0,x)=0,\ \partial_t u_j^{(1)}(0,x) = 0,&&\text{on}\ M.
	\end{aligned}
	\right.
\end{equation*}
The corresponding DtN maps yield $\Lambda_1^{(1)}(h)=\Lambda_2^{(1)}(h)$, $\forall\,h\in \mathcal{H}^{m+1}(\Sigma)$ which means $\partial_{\nu} u_1^{(1)}=\partial_{\nu} u_2^{(1)}$ on $\Sigma$.

Let $v^{(1)}=u_2^{(1)}-u_1^{(1)}$, then
\begin{equation}
	\label{equ:equation_v1}
	\left\{
	\begin{aligned}
		&\partial_t^2 v^{(1)} = a_2\Delta_{g} v^{(1)}+c_{1,2}v^{(1)}+(a_2-a_1)\Delta u_1^{(1)}+(c_{1,2}(x)-c_{1,1}(x))u_{1}^{(1)},\ &&\text{on}\ \mathcal{M},\\
		&v^{(1)}=\partial_{\nu} v^{(1)}=0,\ &&\text{on}\ \Sigma,\\
		&v^{(1)}(0,x)=0,\ \partial_t v^{(1)}(0,x) = 0,&&\text{on}\ M,
	\end{aligned}
	\right.
\end{equation}
and the auxiliary equation is given by
\begin{equation}
	\label{equ:equation_w}
	\left\{
	\begin{aligned}
		&\partial_t^2 w=\Delta_g \left(a_2 w\right)+c_{1,2}w,\ &&\text{on}\ \mathcal{M},\\
		&w(T,x)=0,\ \partial_t w(T,x) = 0,&&\text{on}\ M.
	\end{aligned}
	\right.
\end{equation}
Multiplying both sides of equation \eqref{equ:equation_v1} by $w$ and then integrating by parts, and using the homogeneous boundary conditions of the function and its derivatives, we obtain the following integral identity:
\begin{align}
	\label{equ: first ordert integral}
	L_1 := \int_{\mathcal{M}}  (a_2-a_1)w\Delta u_1^{(1)}+(c_{1,2}(x)-c_{1,1}(x))wu_{1}^{(1)}\mathrm{d}V_{\overline{g}}=0,
\end{align}
where $\mathrm{d}V_{\overline{g}}=|\overline{g}|^{1/2}\mathrm{d}t\mathrm{d}x_1\mathrm{d}x_2\mathrm{d}x_3$
is the volume form on the Lorentz manifold $(\mathcal{M},\overline{g})$.

We first transform the equation $w$ to align with the equation discussed in Section \ref{sec:Gaussian_Beam_Solution}. Let $\widetilde{w} = a_2 w $ satisfy the following equation: 
\begin{equation*}
	\left\{
	\begin{aligned}
		&\partial_t^2 \widetilde{w} =a_2\Delta_g \left(\widetilde{w} \right)+c_{1,2}\widetilde{w} ,\ &&\text{on}\ \mathcal{M},\\
		&\widetilde{w} (T,x)=0,\ \partial_t \widetilde{w} (T,x) = 0,&&\text{on}\ M.
	\end{aligned}
	\right.
\end{equation*}
On any maximal null geodesics $\alpha(t)$ of $\mathcal{M}$, we construct the Gauss beams of order $1$ as follows:
\begin{equation}
	\label{equ:two Gaussian beam}
	\begin{aligned}
		u_1^{(1)} &= A_1^{(1)}(t,x)e^{\mathrm{i}\tau \phi_1^{(1)}}+R^{(1)}_{1,\tau},\\
		\widetilde{w}  &= A_2(t,x)e^{-\mathrm{i}\tau \overline{\phi_2}}+R_{2,\tau},
	\end{aligned}
\end{equation}
satisfying $\phi_1^{(1)} = \phi_2$ and $A_1^{(1)}=A_2=\det(Y(s))^{-1/2}\chi\left(\frac{|z'|}{\delta}\right)$. Substituting \eqref{equ:two Gaussian beam} into \eqref{equ: first ordert integral}, together with the estimate (\ref{equ:Ru_H1_property}) for $R^{(1)}_{1,\tau}$ and $R_{2,\tau}$, we obtain
\begin{align*}
    0=\lim\limits_{\tau\to\infty}\frac{L_1}{\tau^2} &= \lim\limits_{\tau\to\infty} \int_{\mathcal{M}} \left(1-\frac{a_1}{a_2}\right)A_1^{(1)}A_2e^{-2\tau \Im \phi_2}\mathrm{d}V_{\overline{g}}\\
    & = =\lim\limits_{\tau\to\infty}\int_{|z'|<\delta}\int_{\alpha(t_{-}-\delta')}^{\alpha(t_{+}+\delta')} \left(1-\frac{a_1}{a_2}\right)\det(|Y(s)|)^{-1}\chi^2\left(\frac{|z'|}{\delta} \right) e^{-2\tau \Im \phi_2} \mathrm{d}s\wedge \mathrm{d}z'
\end{align*}
It follows from Lemmas \ref{lem:Y},\,\ref{ray_transform} and \cite{FIKO21} that
\begin{align*}
	\mathcal{J}^{(2)}\left(1-\frac{a_1}{a_2}\right) = \int_{t\in I} \left(1-\frac{a_1}{a_2}\right) (\alpha(t))\mathrm{d} t=0,
\end{align*}
which implies $a_1=a_2$ on $\alpha(t)$. For any $p=(t_0,x_0)\in \mathcal{M}$, there is a maximal null geodesics  $\alpha(t)$
passing through $p$, which implies $a_1=a_2$ at $p$. The arbitrariness of $p$ implies $a_1=a_2=a$ on $\mathcal{M}$.

Using the same method as described above, we have
\begin{align*}
0=\lim\limits_{\tau\to\infty}L_1 &=\lim\limits_{\tau\to\infty}\int_{\mathcal{M}}\frac{1}{a_2}(c_{1,2}-c_{1,1})A_1^{(1)}A_2e^{-2\tau \Im \phi_2}\mathrm{d}V_{\tilde{g}}\\
&=\lim\limits_{\tau\to\infty}\int_{|z'|<\delta}\int_{\alpha(t_{-}-\delta')}^{\alpha(t_{+}+\delta')} \frac{1}{a_2}(c_{1,2}-c_{1,1})\det(|Y(s)|)^{-1}\chi^2\left(\frac{|z'|}{\delta} \right) e^{-2\tau \Im \phi_1^{(1)}}\mathrm{d}s\wedge \mathrm{d}z'.
\end{align*}
which implies $c_{1,2}=c_{1,1}$ on $\alpha(t)$ and so that $c_{1,2}=c_{1,1}=c_1$ on $\mathcal{M}$

Next, we show the uniqueness of $b_2$ and $c_2$ in (\ref{eq:g ck}) with similar arguments to $c_1$. Denoted by $u_{j}^{(12)}=\partial_{\epsilon_1\epsilon_2}^2 u_j\big|_{\epsilon=0}$
the second-order linearizations of $u_{j}$ $(j=1,2)$, they satisfy
\begin{equation*}
	\left\{
	\begin{aligned}
		&\partial_t^2 u_{j}^{(12)} =  a(x)\Delta_{g} u_{j}^{(12)}+c_1(t,x) u_{j}^{(12)}+a(x)\Delta_g\left(b_{j,2}(x)u^{(1)}u^{(2)}\right)+c_{j,2}(x)u^{(1)}u^{(2)},\ &&\text{on}\ \mathcal{M},\\
		&u_{j}^{(12)}(x)=0,\ &&\text{on}\ \Sigma,\\
		&u_{j}^{(12)}(0,x)=0,\ \partial_t u_{j}^{(12)}(0,x) = 0, &&\text{on}\ M,
	\end{aligned}
	\right.
\end{equation*}
The corresponding DtN maps yield $\Lambda_1^{(12)}(h)=\Lambda_2^{(12)}(h)$, $\forall\,h\in \mathcal{H}^{m+1}(\Sigma)$ which means $\partial_{\nu} \left(u_1^{(12)}+b_{1,2}(x)u^{(1)}u^{(2)}\right)=\partial_{\nu} \left(u_2^{(12)}+b_{2,2}(x)u^{(1)}u^{(2)}\right)$ on $\Sigma$.

Let $v^{(12)}=u_2^{(12)}-u_1^{(12)}$, then
\begin{equation}
	\label{equ:equation_v12}
	\left\{
	\begin{aligned}
		&\partial_t^2v^{(12)} =  a(x)\Delta_{g} v^{(12)}+c_1(x) v^{(12)}\\
  +&a(x)\Delta_{g}\left((b_{2,2}(x)-b_{1,2}(x))u^{(1)}u^{(2)}\right)+(c_{2,2}(x)-c_{1,2}(x))u^{(1)}u^{(2)},\ &&\text{on}\ \mathcal{M},\\
		&u_{j}^{(12)}(x)=0,\ &&\text{on}\ \Sigma,\\
		&u_{j}^{(12)}(0,x)=0,\ \partial_t u_{j}^{(12)}(0,x) = 0,&&\text{on}\ M,
	\end{aligned}
	\right.
\end{equation}	
Multiplying both sides of equation \eqref{equ:equation_v12} by $w$ and then integrating by parts, and using the homogeneous boundary conditions of the function and its derivatives, we obtain the following integral identity:
\begin{align}
	\label{equ:second order integral}
	L_2 := \int_{\mathcal{M}}\left(a(x)w\Delta_{g}((b_{2,2}-b_{1,2})u^{(1)}u^{(2)})+(c_{2,2}-c_{1,2})u^{(1)}u^{(2)}w\right)\mathrm{d}V_{\tilde{g}}=0.
\end{align}

On any maximal null geodesics $\alpha(t)$ of $\mathcal{M}$, we construct the Gaussian beams of order $2$ as follows:
\begin{equation}
	\label{equ:three Gaussian beam}
	\begin{aligned}
		u^{(1)} &= A_1^{(1)}(t,x)e^{\mathrm{i}\tau \phi_1^{(1)}}+R^{(1)}_{1,\tau},\\
		u^{(2)} &= A_1^{(2)}(t,x)e^{\mathrm{i}\tau \phi_1^{(2)}}+R^{(2)}_{1,\tau},\\
		w &= \frac{1}{a(x)}\overline{A_2}(t,x)e^{-\mathrm{i}\tau \overline{\phi_2}}+\frac{1}{a(x)}R_{2,\tau},
	\end{aligned}
\end{equation}
satisfying $\phi_1^{(1)}=\phi_1^{(2)}=\frac{\phi_2}{2}$ and $A_1^{(1)}=A_1^{(2)}=A_2=\det(Y(s))^{-1/2}\chi^2\left(\frac{|z'|}{|\delta|}\right)$.
Substituting \eqref{equ:three Gaussian beam} into \eqref{equ:second order integral}, together with the estimate in (\ref{equ:Ru_H1_property}) for $R^{(1)}_{1,\tau}$, $R^{(2)}_{1,\tau}$, and $R_{2,\tau}$, we obtain
\begin{align*}
0=\lim\limits_{\tau\to\infty}\tau^{-2}L_2& = \lim\limits_{\tau\to\infty}\int_{\mathcal{M}}  A_1^{(1)}A_1^{(2)}\overline{A_2}e^{-\tau\Im\phi_2}(b_{2,2}-b_{1,2})\langle d\phi_2,d\overline{\phi_2} \rangle_{\overline{g}} \mathrm{d}V_{\overline{g}}\\
& =\int_{t\in I} (b_{2,2}-b_{1,2})(\gamma(t))  \langle d\phi_2,d\overline{\phi_2} \rangle_{\overline{g}}\left|\det(Y(s))\right|^{-1}\det(Y(s))^{-1/2}\mathrm{d}t.
\end{align*}
Base on Lemma\,\ref{lem:Y}, we have $\det(Y)=\det(\widetilde{Y})$, which implies	
\begin{align*}
	\mathcal{J}_{\widetilde{Y}}^{(1)}\left((b_{2,2}-b_{2,1})\langle d\phi_2,d\overline{\phi_2} \rangle_{\overline{g}}\right)=0.
\end{align*}
Together with $\langle d\phi_2,d\overline{\phi_2} \rangle_{\overline{g}}>0$, we obtain $b_{2,1}=b_{2,2}$ on $\mathcal{M}$.

Therefore, the integral $L_2$ can be simplified as
\begin{align*}
	\int_{\mathcal{M}} A_1^{(1)}A_1^{(2)}\overline{A_2}e^{-\tau\Im \phi_2}(c_{2,1}-c_{2,2})\langle d\phi_2,d\overline{\phi_2} \rangle_{\overline{g}} \mathrm{d}V_{\overline{g}}=0,
\end{align*}
implying that $c_{2,1}=c_{2,2}$ on $\mathcal{M}$.

Finally, we prove the uniqueness of $c_k$ for $k\geq 3$ in (\ref{eq:g ck}) via using an induction argument as follows.
Suppose $b_2,c_1,\cdots,c_{k-1}$ have been uniquely determined through the measurements $\Lambda^{(1\cdots (k-1))}(h)$, which implies that we have uniquely determined the first-order to $(k-1)$-th order linearizations of $u_j$. Denoted by $u_j^{(1\cdots k)} = \partial_{\epsilon_1\cdots \epsilon_k}u_j\big|_{\epsilon=0}$
the $k$-order linearizations of $u_{j}$ $(j=1,2)$, they satisfy
\begin{equation*}
	\left\{
	\begin{aligned}
		&\partial_t^2 u_j^{(1\cdots k)} = a(x)\Delta_{g} u_j^{(1\cdots k)}+c_1(t,x) u_j^{(1\cdots k)}+\\&+\frac{1}{(k-1)!}\left[c_{j,k}(x)u^{(1)}\cdots u^{(k)}\right]
		+a(x)\Delta F_j^{(1\cdots k)}+G_j^{(1\cdots k)}=0, &&\text{on}\ \mathcal{M},\\
		&u_j^{(1\cdots k)}(x)=0,\ &&\text{on}\ \Sigma,\\
		&u_j^{(1\cdots k)}(0,x)=0,\ \partial_t u_j^{(1\cdots k)}(0,x) = 0,&&\text{on}\ M,
	\end{aligned}
	\right.
\end{equation*}
The corresponding DtN maps yield $\Lambda_1^{(1\cdots k)}(h)=\Lambda_2^{(1\cdots k)}(h)$, $\forall\,h\in \mathcal{H}^{m+1}(\Sigma)$.

Let $v^{(1\cdots k)}=u_2^{(1\cdots k)}-u_1^{(1\cdots k)}$. Using the induction assumption, it holds that
\begin{equation}
	\label{equ:v1k}
	\left\{
	\begin{aligned}
		&\partial_t^2 v^{(1\cdots k)} = a(x)\Delta_{g} v^{(1\cdots k)}+c_1(t,x) v^{(1\cdots k)}+\\&=\frac{1}{(k-1)!}\left[(c_{2,k}(x)-c_{1,k}(x))u^{(1)}\cdots u^{(k)}\right],
		\ &&\text{on}\ \mathcal{M},\\
		&v^{(1\cdots k)}(x)=0,\ &&\text{on}\ \Sigma,\\
		&v^{(1\cdots k)}(0,x)=0,\ \partial_t v^{(1\cdots k)}(0,x) = 0, &&\text{on}\ M,
	\end{aligned}
	\right.
\end{equation}
Multiplying both sides of equation \eqref{equ:v1k} by $w$ and then integrating by parts, and using the homogeneous boundary conditions of the function and its derivatives, we obtain the following integral identity:
\begin{align}
	\label{equ:k order integral}
	L_k: = \int_{\mathcal{M}} w\left[(c_{2,k}-c_{1,k})u^{(1)}\cdots u^{(k)}\right]\mathrm{d}V_{\overline{g}}=0.
\end{align}

On any maximal null geodesics $\alpha(t)$ of $\mathcal{M}$, we construct the Gaussian beams of order $k$ as follows:
\begin{equation}
	\label{equ:k Gaussian beam}
	\begin{aligned}
		u^{(i)} &= A_1^{(i)}(t,x)e^{i\tau\phi^{(i)}_1}+R_{1,\tau}^{(i)}, \, \, \mbox{for} \, \, i=1,\cdots,k,\\
		w &= \frac{1}{a(x)}\overline{A_2}(t,x)e^{-i\tau \overline{\phi_2}}+\frac{1}{a(x)}R_{2,\tau},
	\end{aligned}
\end{equation}
satisfying $\phi^{(1)}_1=\cdots =\phi_1^{(k)}=\frac{\phi_2}{k}$ and $a_1^{(1)}=\cdots=a_1^{(k)}=\overline{a}_2 = \det(Y(s))^{-1/2}$.
Substituting \eqref{equ:k Gaussian beam} into \eqref{equ:k order integral}, we have
\begin{align*}
	\int_{\mathcal{M}}A_1^{(1)}\cdots A_1^{(k)}\overline{A}_2(c_{k,1}-c_{k,2})e^{-\tau \Im \phi_2}\mathrm{d}V_{\overline{g}}=0,
\end{align*}
which implies $c_{k,1}=c_{k,2}$ on $\mathcal{M}$.

The proof of Theorem~\ref{Theo:IP11} is complete. 


\subsection{Proof of Theorem \ref{Theo:IP12}}

In this section, we provide the proof of Theorem~\ref{Theo:IP12} which adopts and develops the method of stationary phases from  \cite{UZ21a,UZ21b} in a different context. 

First, we present the following stationary phase lemma along with some auxiliary results to determine $b_2$:
\begin{Lemma}
\label{lemma: stationary phase}
(\cite{M20}) Consider the oscillatory integral
$I(\tau)$ defined as
\begin{equation*}
I(\tau) = \int_{\mathbb{R}^3}A(p)e^{\mathrm{i}\tau\phi(p)}dp,
\end{equation*}
and assume that
\begin{enumerate}
\item $A\in C^{6}(\mathbb{R}^3;\mathbb{C})$ and $\phi\in C^{9}(\mathbb{R}^3;\mathbb{R})$;

\item $p$ is the only critical of $\phi$ on supp$(A)$;

\item The Hessian $D^2\phi$ at $p$ is not degenerate.
\end{enumerate}

Then $I(\tau)$ is well-defined and has the asymptotic expansion
\begin{align*}
   I(\tau) = \left(\frac{2\pi}{\tau}\right)^{3/2}\frac{e^{\mathrm{i}\tau \phi(p)+\mathrm{i}\frac{\pi}{4}\text{sgn}(D^2(p))}}{|\det\left(D^2\phi(p) \right)|^{1/2}}\left(A(p)+\mathcal{O}(\tau^{-1})\right), \,\,  \mbox{as} \,\, \tau\to \infty.
\end{align*}
\end{Lemma}
\begin{Lemma}\label{lemma:4.6}(\cite{FO22})
For each point $p\in \mathcal{M}$, there exist four coefficients $\kappa^{(i)}$ for $i=0,1,2,3$ that satisfy
\begin{equation}
	\kappa^{(0)},\kappa^{(1)}>0,\,\, \kappa^{(2)},\kappa^{(3)}<0,\,\,  \sum\limits_{i=0}^{3} \kappa^{(i)}=0,
\end{equation}
and four tangent vectors $\zeta^{(i)}\in T_p M$ for $i=0,1,2,3$ that satisfy
\[
	\zeta^{(0)}\neq \zeta^{(1)},\,\, \sum\limits_{i=0}^{3} \kappa^{(i)}\zeta^{(i)}=0.
\]
\end{Lemma}

Consider the following equation:
\begin{equation}
\label{equ:w_equation_1}
    \left\{
    \begin{aligned}
        &\partial_{t}^2w = \Delta_g (aw)+c_1(t,x)w,\ &&\text{on}\ \mathcal{M},\\
        &w(t,x)=g_0,\ &&\text{on}\ \Sigma=(0,T)\times \partial M,\\
        &w(T,x)=0,\ \partial_t w(T,x) = 0,&&\text{on}\ M,
    \end{aligned}
    \right.
\end{equation}
and consider $\widetilde{w} = aw$ satisfying
\begin{equation}
\label{equ:w_equation_2}
    \left\{
    \begin{aligned}
        &\partial_{t}^2\widetilde{w} = a\Delta_g (\widetilde{w})+c_1(t,x)\widetilde{w},\ &&\text{on}\ \mathcal{M},\\
        &\widetilde{w}(t,x)=ag_0,\ &&\text{on}\ \Sigma=(0,T)\times \partial M,\\
        &\widetilde{w}(T,x)=0,\ \partial_t \widetilde{w}(T,x) = 0,&&\text{on}\ M,
    \end{aligned}
    \right.
\end{equation}
and set $i_1=1$ and $i_2=2$ in equations \eqref{equ:Second_Order_linearization_equation} and \eqref{equ:Second_Order_linearization_DtN}.
By applying integration by parts, we obtain
\begin{align*}
    &\int_0^T\int_{\partial M} ag_0\Lambda^{(12)}_{h}\mathrm{d}s\mathrm{d}t\\
    =&\int_0^T\int_{\partial M} ag_0\partial_{\nu}\left(u^{(12)}+b_2(t,x)u^{(1)}u^{(2)}\right)\mathrm{d}s\mathrm{d}t\\
    =&\int_0^T\int_{M} wa\Delta_g \left(u^{(12)}+b_2(t,x)u^{(1)}u^{(2)}\right)-\left\langle\nabla^g (aw), \nabla^g \left(u^{(12)}+b_2(t,x)u^{(1)}u^{(2)}\right)\right\rangle_g\mathrm{d}V_g\mathrm{d}t\\
    =& \int_0^T\int_{M}w\left(\partial_t^2 u^{(12)}-c_1(t,x)u^{(12)}-c_2(t,x)u^{(1)}u^{(2)}\right)-\left\langle\nabla^g \widetilde{w}, \nabla^g \left(u^{(12)}+b_2(t,x)u^{(1)}u^{(2)}\right)\right\rangle_g\mathrm{d}V_g\mathrm{d}t\\
    =&\int_0^T\int_{M} u^{(12)}\partial_t^2 w-w\left(c_1(t,x)u^{(12)}+c_2(t,x)u^{(1)}u^{(2)}\right)-\left\langle\nabla^g \widetilde{w}, \nabla^g \left(u^{(12)}+b_2(t,x)u^{(1)}u^{(2)}\right)\right\rangle_g\mathrm{d}V_g\mathrm{d}t\\
    =&\int_0^T\int_{M}-c_2(t,x)u^{(1)}u^{(2)} w-\left\langle\nabla^g \widetilde{w}, \nabla^g \left(b_2(t,x)u^{(1)}u^{(2)}\right)\right\rangle_g\mathrm{d}V_g\mathrm{d}t,
\end{align*}
where $V_g = |g|^{1/2}\mathrm{d}x_1\mathrm{d}x_2\mathrm{d}x_3$. This implies
\begin{align}
\label{equ:second_order_intergral_stationary_version}
    I_2 := \int_0^T\int_{M} \left\langle \nabla^g \widetilde{w},\nabla^g \left(b_2(t,x)u^{(1)}u^{(2)}\right) \right\rangle_{g}+c_2(t,x)u^{(1)}u^{(2)}w\mathrm{d}V_{g}\mathrm{d}t=-\int_0^T\int_{\partial M} ag_0\Lambda^{(12)}_{h}\mathrm{d}s\mathrm{d}t,
\end{align}
where the right-hand side is given.

Let $\alpha_1^{(1)}(t)=(t,\gamma_1^{(1)}(t))$, $\alpha_1^{(2)}(t)=(t,\gamma_1^{(2)}(t))$, and $\alpha_2(t)=(t,\gamma_2(t))$
be three distinct maximal null geodesics passing through $p\in \mathcal{M}$. Furthermore, we require these three maximal null geodesics to satisfy the following two properties:
\begin{enumerate}[(P1)]
	\item\label{item_P1} $\alpha_1^{(1)}$, $\alpha_1^{(2)}$, and $\alpha_2$  intersect at only one point $p = (t_0,x_0)\in \mathcal{M}$ for $t_0\in (0,T)$.
	\item\label{item_P2} Let $\xi_1^{(1)}$, $\xi_1^{(2)}$, and $\xi_2$ be the tangent vectors of $\alpha_1^{(1)}$, $\alpha_1^{(2)}$, and $\alpha_2$ at $p$, respectively, satisfying
	\begin{equation}
		\label{equ:definition_xi_1}
		\begin{aligned}
			&\xi_1^{(1)} = \xi^{(0)}=(1,\zeta^{(0)}),\quad \xi_1^{(2)} = \xi^{(1)}=(1,\zeta^{(1)}),\\
			&\xi_2 =\frac{\kappa^{(2)}\xi^{(2)}+\kappa^{(3)}\xi^{(3)}}{\kappa^{(2)}+\kappa^{(3)}} = \left(1,\frac{\kappa^{(2)}\zeta^{(2)}+\kappa^{(3)}\zeta^{(3)}}{\kappa^{(2)}+\kappa^{(3)}}\right),
		\end{aligned}
	\end{equation}
where $\kappa^{(i)}$ and $\zeta^{(i)}$ for $i=1,2,3$ are defined in Lemma \ref{lemma:4.6}. By selecting the coefficients $\kappa_1^{(1)}$, $\kappa_1^{(2)}$, $\kappa_2$ to satisfy
\begin{align*}
	\kappa_1^{(1)} = \kappa^{(0)},\,\,\,\kappa_1^{(2)} = \kappa^{(1)},\,\,  \kappa_2 = -\left(\kappa^{(2)}+\kappa^{(3)}\right),
\end{align*}
we have
\begin{align}
	\label{equ:linear correlation}
	\kappa_1^{(1)}\xi_1^{(1)}+\kappa_1^{(2)}\xi_1^{(2)}-\kappa_2\xi_2 = 0.
\end{align}
\end{enumerate}

On these three maximal null geodesics, we construct the corresponding Gaussian beams as follows:
\begin{equation}
	\label{equ:three_Gaussian_beam_Phase}
	\begin{aligned}
		u^{(1)} &= A_1^{(1)}(x)e^{i\tau \kappa_1^{(1)}\phi_1^{(1)}}+R_{1,\tau}^{(1)}, \,\, \mbox{on} \,\, \alpha_1^{(1)}, \\
		u^{(2)} &= A_1^{(2)}(x)e^{i\tau \kappa_1^{(2)}\phi_1^{(2)}}+R_{1,\tau}^{(2)}, \,\, \mbox{on} \,\, \alpha_1^{(2)}, \\
		\widetilde{w} &= \overline{A_2}(x)e^{-i\tau \kappa_2\overline{\phi_2}}+R_{2,\tau},  \,\, \mbox{on} \,\, \alpha_2. \\
	\end{aligned}
\end{equation}
where $\kappa_1^{(1)}$, $\kappa_1^{(2)}$, and $\kappa_2$ are defined in (P2). Substituting \eqref{equ:three_Gaussian_beam_Phase} into \eqref{equ:second_order_intergral_stationary_version}, together with estimate (\ref{equ:Ru_H1_property}) with $m=1$ for the corresponding remainders, we obtain
\begin{align}
	\label{equ:second order integral phase}
	\lim\limits_{\tau\to +\infty}\tau^{-2}I_2 = \lim\limits_{\tau\to +\infty}\int_{\mathcal{M}} b_2(t,x)A_1^{(1)}A_1^{(2)}\overline{A_2}e^{i\tau (\kappa_1^{(1)}\phi_1^{(1)}+\kappa_1^{(2)}\phi_1^{(2)}-\kappa_2\overline{\phi_2})} \mathrm{d}V_{\overline{g}}.
\end{align}

Let the amplitude and phase terms in the above integral be
\begin{equation}\label{eq:ttt1}
	\mathbb{A}_2=A^{(1)}_1A^{(2)}_1\overline{A_2} \,\, \mbox{and} \,\, \Phi_2 := \kappa_1^{(1)}\phi_1^{(1)}+\kappa_1^{(2)}\phi_1^{(2)}-\kappa_2\overline{\phi_2},
\end{equation}
respectively. We have:
\begin{Lemma}
\label{lemma: Property phase}
The terms in \eqref{eq:ttt1} satisfy the following properties:
\begin{enumerate}
\item	The amplitude $\mathbb{A}_2$ and the phase $\Phi_2$ are sufficiently smooth, satisfying the Assumption 1 in Lemma \ref{lemma: stationary phase}.

\item $\Phi_2(p)=0$ and $\nabla^{\tilde{g}} (\Phi_2)(p)=0$.

\item $D^2\Im \Phi_2(X,X)>0,\forall\,X\in T_{p}\mathcal{M}\setminus 0$.
\end{enumerate}
\end{Lemma}
\begin{proof}
(1) Since $g$ is $C^{11}$ smooth in the Fermi coordinates, we know that $\mathbb{A}_2 \in C^{7}(\mathcal{M})$ and $\Phi_2 \in C^{10}(\mathcal{M})$
by similar arguments in \cite{FO22}. Therefore, the first assumption in Lemma\,\ref{lemma: stationary phase} is satisfied.

(2) Note that the phases $\phi_1^{(j)}$ vanish along $\alpha_1^{(j)}$ for $j=1,2$, and $\phi_2$ vanishes along $\alpha_{2}$.
Therefore, at the intersection point $p$, we have $\phi_1^{(1)}(p)=\phi_1^{(2)}(p)=\phi_2(p)=0$ and $\Phi_2(p)=0$.
Then, using $\nabla^{\bar{g}}\phi_1^{(j)}\left|_{\alpha_j}\right.=\dot{\alpha}_1^{(j)}$ and $\nabla^{\bar{g}}\phi_2\left|_{\alpha_2}\right.=\dot{\alpha}_2$, along with equations \eqref{equ:definition_xi_1} and \eqref{equ:linear correlation}, we find $\nabla^{\tilde{g}}(\Phi_2)(p)=0$.
	
(3) Using the Fermi coordinates, for each $j=1,2$, we have
\begin{align*}
	&D^2\Im \phi_1^{(j)}(X,X)\geq 0,\quad  D^2\Im \phi_2(X,X)\geq 0,\quad \forall X\in T_p\mathcal{M},\\
	&D^2\Im \phi_1^{(j)}(X,X)> 0, \quad \forall X\in T_p\mathcal{M}\setminus \text{span}\ \xi_1^{(j)},\\
	&D^2\Im \phi_2(X,X)> 0, \quad \forall X\in T_p\mathcal{M}\setminus \text{span}\ \xi_2,
\end{align*}
since $\xi_1^{(1)}$ and $\xi_1^{(2)}$ are linearly independent. Furthermore, we have $D^2\Im \Phi_2(X,X)>0$ for all $X\in T_{p}\mathcal{M}\setminus 0$.
\end{proof}

It follows from $b=b_2\in C^{\infty}(\mathcal{M})$ and property 1 of Lemma\,\ref{lemma: Property phase} that
$b_2 \mathbb{A}_2\in C^7(\mathbb{R}\times\mathbb{R}^3;\mathbb{C})$. This implies the satisfaction of the first assumption in Lemma\,\ref{lemma: stationary phase}.
Moreover, properties 2 and 3 of Lemma\,\ref{lemma: Property phase} ensure that the second and third assumptions in Lemma\,\ref{lemma: stationary phase} are also satisfied.
Therefore, the right-hand side of (\ref{equ:second order integral phase}) has the asymptotic expansion
\begin{equation}
	\label{equ:second order phase stationary lemma}
\begin{aligned}
	&\int_{\mathcal{M}} b_2(x) \mathbb{A}_2e^{\mathrm{i}\tau \Phi_2}\mathrm{d}V_{\overline{g}} \\
	= &\left(\frac{2\pi}{\tau}\right)^{2}\frac{e^{\mathrm{i}\tau \Phi_2(p)+\mathrm{i}\frac{\pi}{4}sgn(D^2\Phi_2(p))}}{|\det(D^2\Phi_2(p))|^{1/2}}(b_2(p) A_1^{(1)}(p)A_1^{(2)}(p)\overline{A_2}(p))+\mathcal{O}(\tau^{-3}),
\,\,  \mbox{as} \,\, \tau\to \infty.
\end{aligned}
\end{equation}
By combining (\ref{equ:second_order_intergral_stationary_version}), \eqref{equ:second order integral phase}, \eqref{equ:second order phase stationary lemma}, and the fact $\Phi_2(p)=0$, we have
\begin{align*}
	\lim\limits_{\tau\to \infty} I_2=\left(2\pi\right)^{2}\frac{e^{\mathrm{i}\frac{\pi}{4}sgn(D^2\Phi_2(p))}}{|\det(D^2\Phi_2(p))|^{1/2}}(b_2(p) A_{1,0}^{(1)}(p)A_{1,0}^{(2)}(p)\overline{A_{2,0}}(p)),
\end{align*}
where $A_{1,0}^{(1)}$, $A_{1,0}^{(2)}$, and $A_{2,0}$ are the dominant terms with respect to $\tau$ of $A_{1}^{(1)}$, $A_{1}^{(2)}$, and $A_{2}$, respectively.
Given the right-hand side of (\ref{equ:second_order_intergral_stationary_version}), we can uniquely determine the following expression
\begin{equation}\label{eq:rb1}
	b_2(p)A_{1,0}^{(1)}(p)A_{1,0}^{(2)}(p)\overline{A_{2,0}}(p),
\end{equation}
where $A_{1,0}^{(1)},A_{1,0}^{(2)},A_{2,0}$ are constructed as shown in \eqref{equ:amplitude_form}. Consequently,
$b_2(p)$ can be computed uniquely. Since $p$ can be any point in $\mathcal{M}$, this completes the determination of $b_2$ in $\mathcal{M}$.

To construct the coefficient $c_2$, we begin by considering the integral defined by $c_2$ as
\begin{equation}\label{eq:rc1}
\begin{split}
	& \lim\limits_{\tau\to\infty} \left(I_2-\int_{\mathcal{M}} \left\langle \nabla^g \widetilde{w},\nabla^g \left(b_2(t,x)u^{(1)}u^{(2)}\right) \right\rangle_{g}\mathrm{d}V_{\overline{g}}\right)\\
	=& \lim\limits_{\tau\to\infty}\int_{\mathcal{M}} c_2(t,x)u^{(1)}u^{(2)}w\mathrm{d}V_{\overline{g}}\\
	=& \lim\limits_{\tau\to\infty}\int_{\mathcal{M}}\frac{1}{a}c_2(t,x)A_1^{(1)}A_1^{(2)}\overline{A_2}e^{\mathrm{i}\tau \Phi_2}   \mathrm{d}V_{\overline{g}}. 
\end{split}
\end{equation}
Applying the stationary phase method and utilizing Lemma~\ref{lemma: Property phase}, we obtain
\begin{equation}\label{eq:rc2}
\begin{split}
	&\lim\limits_{\tau\to \infty} \tau^2\int_{\mathcal{M}}\frac{1}{a}c_2(t,x)A_1^{(1)}A_1^{(2)}\overline{A_2}e^{\mathrm{i}\tau \Phi_2}   \mathrm{d}V_{\overline{g}}\\
	=& \left(2\pi\right)^{2}\frac{e^{\mathrm{i}\frac{\pi}{4}sgn(D^2\Phi_2(p))}}{|\det(D^2\Phi_2(p))|^{1/2}}\left(\frac{c_2(p)}{a(p)} A_{1,0}^{(1)}(p)A_{1,0}^{(2)}(p)\overline{A_{2,0}}(p)\right).
\end{split}
\end{equation}
Following similar reasoning as for $b_2$, we can recover $c_2$.

Finally, to determine $c_k$ for $k\geq 3$ in (\ref{eq:g ck}), we employ induction.
Assuming that $b_2,c_1,\cdots,c_{k-1}$ for $k\geq 3$ have been uniquely constructed, we can proceed as follows.
For the coefficient $c_k$ with $k\geq 3$, we consider the $k$-th order linearization of the DtN map $\Lambda^{(1\cdots k)}$ as defined in (\ref{equ:k-th_Order_linearization_DtN}). This determines the integral
\begin{align}
\label{equ:Intergral k order}
I_k := \int_{\mathcal{M}}c_k(t,x)wu^{(1)}\cdots u^{(k)}\mathrm{d}V_{\overline{g}}.
\end{align}
Again, we let $\alpha_1^{(1)}(t)=(t,\gamma_1^{(1)}(t))$, $\alpha_1^{(2)}(t)=(t,\gamma_1^{(2)}(t))$, and $\alpha_2(t)=(t,\gamma_2(t))$
be three distinct maximal null geodesics passing through $p\in \mathcal{M}$. We further assume that these three geodesics satisfy properties (P1) and (P2), while the coefficients are defined as
\begin{align*}
    &\kappa_1^{(1)} = \kappa^{(0)},\,\,\,\kappa_2 = -(\kappa^{(2)}+\kappa^{(3)}),\\
    &\kappa_1^{(j)} = \frac{1}{k-2}\kappa^{(1)},\,\,\,j=2,\cdots,k-1,
\end{align*}
where $\kappa^{(0)},\kappa^{(1)},\kappa^{(2)}$ and $\kappa^{(3)}$ are defined in Lemma \ref{lemma:4.6}.
On these three maximal null geodesics, we construct the corresponding Gaussian beams as follows:
\begin{equation}
	\label{equ:k_Gaussian_beam_Phase}
	\begin{aligned}
		u^{(1)} &= A_1^{(1)}(x)e^{\mathrm{i}\tau \kappa_1^{(1)}\phi_1^{(1)}}+R_{1,\tau}^{(1)}, \,\, \mbox{on} \,\, \alpha_1^{(1)}, \\
		u^{(j)} &= A_1^{(j)}(x)e^{\mathrm{i}\tau \kappa_1^{(j)}\phi_1^{(j)}}+R_{1,\tau}^{(j)}, \,\, \mbox{on} \,\, \alpha_1^{(2)}, \,\, \mbox{for}\,\, j=2,\cdots,k,\\
		w &= \frac{1}{a}\overline{A_2}(x)e^{-\mathrm{i}\tau \kappa_2\overline{\phi_2}}+R_{2,\tau},  \,\, \mbox{on} \,\, \alpha_2. \\
	\end{aligned}
\end{equation}
Substituting \eqref{equ:k_Gaussian_beam_Phase} into \eqref{equ:Intergral k order}, together with estimate (\ref{equ:Ru_H1_property}) with $m=1$ for the corresponding remainders, we obtain
\begin{align}
	\label{equ:second order integral phase}
	\lim\limits_{\tau\to +\infty}I_k = \lim\limits_{\tau\to +\infty}\int_{\mathcal{M}} c_k(x)\overline{A_2}\prod_{j=1}^{k}A_1^{(j)}e^{\mathrm{i}\tau \left(\sum\limits_{j=1}^k\kappa_1^{(j)}\phi_1^{(j)}-\kappa_2\overline{\phi_2}\right)} \mathrm{d}V_{\overline{g}}.
\end{align}

Let the amplitude and phase terms be
\begin{equation}\label{eq:ttt2}
	\mathbb{A}_k=\overline{A_2}\prod_{j=1}^{k}A_1^{(j)} \,\, \mbox{and} \,\, \Phi_k :=\sum\limits_{j=1}^k\kappa_1^{(j)}\phi_1^{(j)}-\kappa_2\overline{\phi_2},
\end{equation}
respectively. As demonstrated in the proof of Lemma \ref{lemma: Property phase}, we can establish the following lemma:
\begin{Lemma}
\label{lemma:k Property phase}
The two terms in \eqref{eq:ttt2} satisfy the following properties:
\begin{enumerate}
\item The amplitude $\mathbb{A}_k$ and the phase $\Phi_k$ are sufficiently smooth, and they satisfy the Assumption 1 in Lemma \ref{lemma: stationary phase}.

\item $S_k(p)=0$ and $\nabla^{\tilde{g}} (S_k)(p)=0$.

\item $D^2\Im S_k(X,X)>0,\forall X\in T_{p}\mathcal{M}\setminus 0$.
\end{enumerate}
\end{Lemma}

Similar to the arguments for $b_2$, it follows from Lemmas\,\ref{lemma:k Property phase} and \ref{lemma: stationary phase}
that $c_k$ in (\ref{equ:second order integral phase}) can be uniquely recovered.

\subsection{Proofs of Theorems \ref{Theo:IP21} and \ref{Theo:IP22}.}

For the case of non-zero initial values, we can employ a similar linearization method.
Specifically, we suppose that $u$ satisfies the hyperbolic system \eqref{equ:Quasilinear_equation} where $h$ has the form \eqref{equ:h_expansion}, and let $\widetilde{u}(t,x,\epsilon)$ satisfy the following system:
\begin{equation*}
\left\{
\begin{aligned}
    &\partial^2_t \widetilde{u} = a(x)\Delta_g (\widetilde{u}+F(x,\widetilde{u}))+G(x,\widetilde{u}),\ &&\text{on}\ \mathcal{M},\\
    &\widetilde{u}(t,x)=0,\ &&\text{on}\ \Sigma,\\
    &\widetilde{u}(0,x)=\varphi(x),\ \partial_t \widetilde{u}(0,x) = \psi(x),&&\text{on}\ M.
\end{aligned}
\right.
\end{equation*}
Here, the initial conditions satisfy
\begin{align*}
    \|\varphi\|_{H^{m+1}(M)}+\|\psi\|_{H^m(M)}\leq \frac{\epsilon_0}{2},
\end{align*}
where $\epsilon_0$ is defined in Theorem \ref{theorem:local well-posedness}.
Let
\begin{align*}
    \widetilde{u}^{(i)}(t,x) = \lim\limits_{\epsilon\to 0}\frac{u(t,x)-\widetilde{u}(t,x)}{\epsilon_i},
\end{align*}
which satisfies the system:
\begin{equation*}
\left\{
\begin{aligned}
    &\partial_t^2 u^{(i)}=a(x)\Delta_{g} \widetilde{u}^{(i)}+c_1(x) \widetilde{u}^{(i)},\ &&\text{on}\ \mathcal{M},\\
    &\widetilde{u}^{(i)}(t,x)=h^{(i)},\ &&\text{on}\ \Sigma,\\
    &\widetilde{u}^{(i)}(0,x)=0,\ \partial_t \widetilde{u}^{(i)}(0,x) = 0,&&\text{on}\ M.
\end{aligned}
\right.
\end{equation*}
Thus, $\widetilde{u}^{(i)}$ and $u^{(i)}$ satisfy the same equation, and consistent results hold for higher-order linearizations.
Therefore, we can recover all coefficients $b_2$ and $c_k$ for $k\geq 1$ using the DtN mappings as demonstrated in the proofs of Theorem  \ref{Theo:IP11} and \ref{Theo:IP12}.

Next, we present the proof of the initial value inversion using the observability inequality \cite{DZZ08}. First, we denote $u_j$ for $j=1,2$ as the solutions to the following system:
\begin{equation}
    \left\{
    \begin{aligned}
        &\partial^2_t u_j= a(x)\Delta_g (u_j+F(x,u_j))+G(x,u_j),\ &&\text{on}\ \mathcal{M},\\
        &u_j(t,x)=h,\ &&\text{on}\ \Sigma,\\
        &u_j(0,x)=\varphi_j(x),\ \partial_t u_j(0,x) = \psi_j(x),&&\text{on}\ M,
    \end{aligned}
\right.
\end{equation}
where $F(t,x,\cdot)$ and $G(t,x,\cdot)$ have been previously inverted, and $u_1$ and $u_2$ share the same measurements as
\begin{align*}
    \Lambda^1_{\varphi_1,\psi_1,F,G}(h) = \Lambda^1_{\varphi_2,\psi_2,F,G}(h), \,\, \forall \,h\in \mathcal{H}^{m+1}(\Sigma).
\end{align*}
 Set $v = u_1-u_2$ to satisfy
\begin{equation}
    \left\{
    \begin{aligned}
        &\partial^2_t v= a(x)(1+K_1)\Delta_g (v)+\left\langle \nabla^g K_1,\nabla^g v\right\rangle_{g}+(\Delta_g K_1+K_2) v,\ &&\text{on}\ \mathcal{M},\\
        & v(t,x)=\partial_{\nu} v(t,x)=0,\ &&\text{on}\ \Sigma,\\
        &v(0,x)=\varphi_1(x)-\varphi_2(x),\ \partial_t v(0,x) = \psi_1(x)-\psi_2(x),&&\text{on}\ M,
    \end{aligned}
\right.
\end{equation}
where
\begin{align*}
    K_1(t,x) = \frac{F(t,x,u_1)-F(t,x,u_2)}{u_1-u_2},\,\,\, K_2(t,x) = \frac{G(t,x,u_1)-G(t,x,u_2)}{u_1-u_2}.
\end{align*}
By the admissible property of $F$ and $G$, and the boundedness of $u_1$ and $u_2$ stated in Remark \ref{rmk:sobolev_embedding}, we have $K_1,\nabla K_1,\Delta K_1,K_2 \in L^{\infty} (\mathcal{M})$.
Therefore,  according to the observability inequality, there exists a constant $C>0$ such that
\begin{align*}
    \|\varphi_1-\varphi_2\|_{H_0^1(\Omega)}+\|\psi_1-\psi_2\|_{L^2(\Omega)}\leq C \|\partial_{\nu} v\|_{L^2(\Sigma)}=0,
\end{align*}
which implies $\varphi_1 = \varphi_2$ and $\psi_1=\psi_2$. This completes the proofs of Theorems \ref{Theo:IP21} and \ref{Theo:IP22}.


\section*{Acknowledgments}

The work of H. Liu was supported by NSFC/RGC Joint Research Scheme (project N\_CityU101/21), ANR/RGC Joint Research Scheme (project A\_CityU203/19) and the Hong Kong RGC General Research Funds (projects 11311122, 11304224 and 11300821). The work of K. Zhang is supported in part by China Natural National Science Foundation (No.~12271207), and by the Fundamental Research Funds for the Central Universities, China.

\end{document}